\documentclass[11pt]{article}
\usepackage{graphicx}
\usepackage[utf8]{inputenc}

\usepackage[cmex10]{amsmath}
\usepackage{amsmath,amssymb, amsthm}
\usepackage{multirow}
\usepackage{authblk}
\usepackage{url}
\usepackage{pgf}
\usepackage{color}
\usepackage{todonotes}
\usepackage{booktabs} 
\usepackage{algorithm}
\usepackage{algorithmic} 
\usepackage{pgfplots}
\usepackage{subcaption}
\usepackage{soul}
\usepackage{bbm}

\newtheorem{prop}{Proposition}
\newtheorem{thm}{Theorem}
\newtheorem{lem}{Lemma}

\usepackage[numbers]{natbib}
 
\usepackage{geometry}
\usepackage{setspace}
\allowdisplaybreaks
\newcommand{\Rho}{\mathrm{P}}

\begin{document}

\title{Distributionally Robust Facility Location Problem under Decision-dependent Stochastic Demand} 

\author{Beste Basciftci\thanks{H. Milton Stewart School of Industrial and Systems Engineering, Georgia Institute of Technology, Email: {\tt beste.basciftci@gatech.edu};}~~~Shabbir Ahmed\thanks{H. Milton Stewart School of Industrial and Systems Engineering, Georgia Institute of Technology;}~~~Siqian Shen\thanks{Department of Industrial and Operations Engineering, University of Michigan at Ann Arbor, Email: {\tt siqian@umich.edu};}}

\date{}

\maketitle

\abstract{Facility location decisions significantly impact customer behavior and consequently the resulting demand in a wide range of businesses. Furthermore, sequentially realized uncertain demand enforces strategically determining locations under partial information. To address these issues, we study a facility location problem where the distribution of customer demand is dependent on location decisions. We represent moment information of stochastic demand as a piecewise linear function of facility-location decisions. Then, we propose a decision-dependent distributionally robust optimization model, and develop its exact mixed-integer linear programming reformulation. We further derive valid inequalities to strengthen the formulation. We conduct an extensive computational study, in which we compare our model with the existing (decision-independent) stochastic and robust models. Our results demonstrate superior performance of the proposed approach with remarkable improvement in profit and quality of service by extensively testing problem characteristics, in addition to computational speed-ups due to the formulation enhancements. These results draw attention to the need of considering the impact of location decisions on customer demand within this strategic-level planning problem.}

~\\
{\bf Keywords:} Facilities planning and design; Distributionally robust optimization; Decision-dependent uncertainty; Integer programming

\section{Introduction}

Determining facility locations has been a fundamental problem in managerial-level decision making for modern transportation and logistics systems. In the most classical setting, a decision maker determines a subset of locations from a given set of candidate sites to open facilities, while assigning customer demand to these locations and minimizing related cost. 
The customer demand plays a critical role in this regard as it is a driving factor in determining where to open facilities. In various types of businesses, location decisions inherently impact customer demand as well. In particular, the availability of nearby facilities can affect customer behavior by boosting demand, which needs to be taken into account in this strategic-level facility location planning. 

The facility location decisions significantly affect customer demand in many settings and consequently determine the success of a business. For instance, in carsharing businesses such as Zipcar and Car2go, customers can choose from a wide range of vehicles to rent for a short period of time, while having fewer cost and responsibilities as compared to full ownership of the vehicles  \citep{Shaheen1998}. Customers can pick up and drop off vehicles from certain rental locations in parking lots, referred to as \textit{stations}, and the convenience of doing so determines the quality of carsharing service. \citet{Jorge2013, Ciari2014} demonstrate that customer demand is affected mainly by the distance to station locations, and therefore their choices of whether or not to use carsharing are significantly impacted by the service availability within their neighborhood \citep[see][]{LeVine2011, Boldrini2016}. Additionally, \citet{Shaheen2006} indicate that when customers observe more vehicle availability and usage convenience after new rental stations open, their confidence to the carsharing service increases, resulting in higher demand for the service.

In addition to carsharing, similar impacts are observed from problems of warehouse location selection on customer demand in supply chains \citep{PengKuan1995}. \citet{Erlenkotter1977} introduces price-sensitive demand relationship in the facility location problem, where demands are associated with facility locations and pricing strategy. 
These studies highlight the importance of location decisions on customer behavior, and present  the need to integrate this decision-dependent demand information in the strategic planning phase of locating facilities. 

On the other hand, customer demands are random and unknown when planning facility locations \citep{Mai1981}. In most cases, the decision maker does not have sufficient  information about the underlying demand. One approach is to consider demand forecasts when planning, and then solve a deterministic facility location problem using estimated demand values. Although a deterministic model is easier to handle from solving perspective, it produces inaccurate results by not fully capturing the underlying uncertainty. Stochastic and robust optimization approaches help the decision maker in this regard, depending on how much information we know about the uncertainty. Although stochastic optimization techniques are powerful in modeling various settings, their applicability is limited to the cases in which the distribution of the underlying uncertainty is fully known. Distributionally robust optimization (DRO) approaches, on the other hand, address this issue, and can obtain robust solutions under partial information. 

In this paper, we propose a distributionally robust facility location problem under demand uncertainty by considering the dependency between customer demand and facility location decisions. To the best of our knowledge, this is the first study that formally considers the impact of location decisions on demand uncertainty within this strategic-level planning problem. 
Below we summarize the contributions of this work. 
\begin{enumerate}
\item We describe moment information of random customer demand by piecewise linear functions of facility-location decisions, highlighting the interplay between location choices and customer behavior. 
\item We formulate a decision-dependent distributionally robust facility location model. We obtain an exact mixed-integer linear programming reformulation of the studied model through duality and convex envelopes. We then propose enhancements to strengthen the formulation, including the derivation of valid inequalities.
\item We conduct numerical studies to demonstrate the effectiveness of the proposed distributionally robust decision-dependent approach from various perspectives. We develop a framework to compare our approach with the existing stochastic and robust optimization methodologies neglecting the decision-dependency. Our results highlight significant increase in profit and reduction in unmet demand, and its robust performance under numerous problem characteristics along with the computational efficiency of the formulation enhancements. 
\end{enumerate}

The remainder of this paper is organized as follows. In Section \ref{litReview}, we review the relevant literature in optimization, transportation, and logistics. In Section \ref{Modelformulation}, we express demand distribution as a function of location decisions, and then present  the decision-dependent distributionally robust facility location problem along with several formulation enhancements. In Section \ref{Computations}, we present our computational studies on a variety of randomly generated instances to evaluate the performance of the proposed approach under different settings. Section \ref{Conclusion} concludes the paper with final remarks and future research directions.

\section{Literature Review}
\label{litReview}

\paragraph{Facility location problem variants.} 
The facility location problem has been studied and analyzed for a wide variety of applications \citep[see, e.g.,][]{Owen1998, Melo2009} to determine the locations of warehouses \citep{Ozsen2008}, distribution centers \citep{Zhang2015}, emergency medical services \citep{Chen2016}, etc. Given the emergence of the Internet of Things (IoT), this problem is also considered for building smarter and connected cities via the determination of the optimal locations of sensors and devices to enhance data flows \citep{Fischetti2017}.  An important branch of facility location studies considers uncertainties in problem parameters such as demand. \citet{Snyder2006} provides a thorough review of facility location problems under various uncertainties in demand and cost parameters, and facility characteristics. 

\paragraph{Stochastic programming and DRO methods.} 
Stochastic programming approaches can be applied to address the issue of uncertain parameter once we know the full distributional information. For example, \citet{Santoso2005} consider a stochastic facility location problem by sampling realizations of demand and capacity parameters from a certain distribution, and they propose an accelerated Benders decomposition algorithm. 
DRO provides an alternative approach to solve problems under uncertainty when the decision maker has partial information about the distribution. Although DRO approach yields reliable and cost-efficient solutions, it is less studied in the context of facility location.  \citet{Lu2015} consider distributionally robust reliable facility location problem by optimizing over worst-case distributions based on a given distribution of random facility disruptions; \citet{Santivanez2018} generalize the study in \citet{Lu2015} by ensuring a minimum service level in satisfying demand under each disruption scenario. Recently, most relevant to our work, \citet{Liu2019} study DRO for optimally locating emergency medical service stations under demand uncertainty. To ensure the reliability of their plan, joint chance constraints are introduced and a moment-based ambiguity set is used for representing the uncertainty in demand. However, these studies pose limitation of modeling by not capturing the possible impact of location decisions on the uncertain parameters. 

The DRO literature can be classified by the ambiguity set being used to describe distributional information of uncertainties. One line of research considers statistical distance-based ambiguity sets, which consider distributions within a certain distance to a target distribution. Some of the most studied distance measures in this area are $\phi$-divergence \citep{Bental2013, Jiang2016}, Wasserstein distance \citep{Esfahani2018, Gao2016} and Levy-Prokhorov metric \citep{Erdogan2006}. Despite of the extensive literature and computational tractability results involving distance-based ambiguity sets, these approaches require a nominal distribution describing the underlying uncertainty with a high confidence level, which might not be available at the facility planning phase for many applications, especially when the service is newly launched and there exists no prior customer data. Another line of research is moment-based ambiguity sets \citep{Popescu2007, Delage2010}. Depending on the definition of these ambiguity sets and benefiting from duality results, the related DRO models can be formulated as mixed-integer linear or semidefinite programs. In this paper, we focus on moment-based ambiguity sets in our analyses as they enable the representation of the direct effect of location decisions on the moments of random demand, and do not require a target distribution for parameterizing the ambiguity. 

\paragraph{Modeling decision-dependent uncertainty. } 
Integrating decision-dependent uncertainties within an optimization framework involves modeling challenges and computational complexities. The studies in this area can be categorized into two groups. The first group focuses on decisions impacting the time of information discovery. \citet{Goel2006} propose a mixed-integer disjunctive programming formulation for incorporating the relationship between the underlying stochastic processes and decisions affecting the time that uncertainty is revealed. As solving this problem involves computational challenges, \citet{Vayanos2011} propose a decision rule approximation to ensure its tractability. Recently, \citet{Basciftci2019_Adaptive} formulate a generic mixed-integer linear program for finite stochastic processes, and provide structural results specifically on the time of information discovery for each resource of the capacity expansion planning problem along with approximation algorithms. On the other hand, in the second group of studies, decisions change the distribution of the underlying uncertainty, which is the focus of this paper. For example, in stochastic programming context, \citet{Ahmed2000} considers network design problem under uncertainties dependent to design decisions, whereas \citet{Basciftci2019} model generators' failure probabilities dependent on their maintenance and operational plans. \citet{Hellemo2018} conduct an overview of recent studies in this area by providing ways to model decision-dependent uncertainties in stochastic programs. 

In the robust optimization literature, decisions affecting the distributions of underlying  uncertainties have been incorporated to the definition of the uncertainty set. \citet{Nohadani2018} study robust linear optimization problems where the uncertainty set is a function of decision variables, and they derive reformulations under specific cases. Similarly, robust decision-dependent optimization problems are considered within several application areas including software partitioning problem \citep{Spacey2012}, radiotherapy planning \citep{Nohadani2017} and offshore oil planning problems \citep{Lappas2017}. However, decision-dependency has not been fully explored within the DRO framework. \citet{Zhang2016} consider generic decision-dependent distributionally robust problems with moment constraints and focus on demonstrating stability of the optimal solutions, whereas \citet{Royset2017} consider these problems under a class of distance-based ambiguity sets to derive convergence results. Most recently, \citet{Noyan2018, Luo2018} provide nonconvex reformulations for DRO problems under various forms of decision-dependent ambiguity sets. Although these studies provide alternative reformulations, the resulting models need further analyses and require the development of efficient solution algorithms. Additionally, the effect of adopting decision-dependent DRO methods in comparison to the existing stochastic or robust methodologies are not quantitatively verified in these studies. 

Despite of this extensive literature in related domains, the impact of facility location decisions on customer behavior and demand uncertainty has not been formally considered to strategically determine locations to open facilities. Furthermore, decision-dependent DRO has not been studied within the facility location context to derive tractable reformulations. We further demonstrate the advantages of our approach via providing comparisons to the alternative decision-independent approaches, and drawing insights from our computational studies. In brief, our paper proposes a novel approach in determining optimal facility locations while considering decision-dependent demand uncertainties by addressing various gaps in the literature of facility location and optimization under uncertainty.  

\section{Problem Formulation}
\label{Modelformulation}

We present a distributionally robust facility location problem, where the facility location decisions affect the underlying demand distribution of each customer site. We first introduce the ambiguity set for describing the distributional information of demand in Section~\ref{ambiguitySetSection}. Then, we formulate the decision-dependent DRO model and propose reformulation techniques to obtain a single-level mixed-integer linear program in Section~\ref{optimizationModelSection}. To strengthen the obtained formulation, we provide a polyhedral study to derive valid inequalities in Section~\ref{validInequalitiesSection}.

\subsection{Ambiguity set formulation} 
\label{ambiguitySetSection}

Consider a set of possible locations $i \in I$ for building facilities and customer sites $j \in J$ that generate demand. Define binary decision variables $y_i, \ i \in I$ to indicate location decisions, such that $y_i$ is 1 if a facility is open at location $i$, and 0 otherwise. The demand at each customer site $j \in J$ is represented by a random variable $d_j(y)$ whose distribution depends on decision vector $y =[y_i, \ i \in I]^{\mathsf T}$. We consider the case where only mean and variance information are given for the demand distribution. Based on a given set of sample points $\{d^n\}_{n=1}^N$ of demand under the case with no facility allocation, we estimate the base-case mean and variance at each site $j$ as $\bar{\mu}_j = 1/N \sum_{n=1}^N d_j^n$ and $\bar{\sigma}_j^2 = 1/(N-1) \sum _{n=1}^N (d_j^n - \bar{\mu}_j)^2$, respectively.  

Suppose that demand distribution comes from a set of distributions with finite support, where the demand at each customer site $j \in J$ can take values from the set $\mathcal{K} = \{d_1, \cdots, d_K\}$ with probabilities $\pi_{j1}, \cdots, \pi_{jK}$. Specifically, the ambiguity set $U(y) = U(y, \mathcal{K}, \bar{\mu}, \bar{\sigma}, \epsilon^\mu, \underline{\epsilon}^\sigma, \overline{\epsilon}^\sigma)$ 
is given by:
\begin{equation} \label{eq:ambiguitySet}
\begin{split}
U(y) = \Biggl\{ \pi_j \in \mathbb{R}^{|K|}_+ : & \sum_{k =1}^K \pi_{jk} = 1 \quad \forall j \in J, \\
& \bigg|\sum_{k =1}^K \pi_{jk} d_k - \mu_j(y) \bigg| \leq \epsilon^\mu_j \quad \forall j \in J, \\ 
& (\sigma_j^2(y) + (\mu_j(y))^2) \underline{\epsilon}^\sigma_j \leq \sum_{k =1}^K \pi_{jk} d_k^2 \leq (\sigma_j^2(y) + (\mu_j(y))^2) \overline{\epsilon}^\sigma_j \quad \forall j \in J \Biggr\}, 
\end{split}
\end{equation}
where $\mu_j(y)$ and $\sigma_j^2(y)$ are the mean and variance of site $j$'s demand  depending on decision $y$, respectively. The constraints in set \eqref{eq:ambiguitySet} guarantee that (i) the probabilities at all customer sites over the support set sum up to 1, (ii) the mean of $d_j(y)$ is within an $\ell_1$-based distance $\epsilon^\mu_j$ to the mean $\mu_j(y)$, and (iii) the corresponding second moments of $d_j(y)$ is bounded by the sum of $(\mu_j(y))^2$ and $\sigma_j^2(y)$ with upper- and lower-bound parameters satisfying $0 \leq  \underline{\epsilon}^\sigma_j \leq 1 \leq  \overline{\epsilon}^\sigma_j$. Parameters $\epsilon^\mu_j$, $\underline{\epsilon}^\sigma_j$, $\overline{\epsilon}^\sigma_j$ determine the robustness of the ambiguity set for each customer site $j \in J$. Specifically, if we have the perfect knowledge regarding the first and second moments of random demand at site $j$, then $\epsilon_j^\mu = 0$, and $\underline{\epsilon}^\sigma_j$ = $\overline{\epsilon}^\sigma_j$ = 1. Otherwise, we can adjust these parameters to consider distributions within certain proximity to the desired moment information, which consequently impacts the conservativeness of facility location decisions. 

We assume that the demand at site $j$ increases when new facilities are opened in its neighborhood. However, due to the size and capacity of a market, the increase in demand is restricted by an upper-bound value, denoted as $\mu_j^{UB}$ for each site $j$, which can be estimated by considering the growth potential of a market of interest within the planning horizon. Moreover, we assume that the highest variance of demand at a customer site occurs when there is no available facility in its neighborhood, and set it equal to the empirical variance $\bar{\sigma}_j^2$. As the number of facilities in the neighborhood of a customer site increases, the variance of the demand at that site decreases. However, the variance cannot be less than a pre-determined lower-bound value, denoted as $(\sigma_j^{LB})^{2}$ for site~$j$, because of the inherent uncertainty in the market. 

The above assumptions are supported by \citet{Shaheen2006, Hernandez2010}, who demonstrate the increase in customers' confidence based on their past experiences with the provided service and its more availability. Consequently, increased customer confidence is associated with increase in the mean and decrease in the variance of customer demand. We interpret the mean and variance information using piecewise linear functions of the decision variable $y$ as follows to indicate these relations: 
\begin{equation} \label{eq:momentFunctions}
\begin{split}
\mu_j(y) & = \min \left\{\bar{\mu}_j (1 + \sum_{j' \in I} \lambda^\mu_{jj'} y_{j'}), \mu_j^{UB}\right\}, \\
\sigma_j^2(y) & = \max \left\{\bar{\sigma}_j^2 (1 - \sum_{j' \in I} \lambda^\sigma_{jj'} y_{j'}), (\sigma_j^{LB})^{2}\right\}.
\end{split}
\end{equation}  
In \eqref{eq:momentFunctions}, the effect of the distance of different facility locations on demand at a target customer site $j$ is controlled by parameters $\lambda_j^\mu, \ \lambda_j^\sigma \in [0,1]^{|I|}$, in such a way that closer locations can have higher impacts on the first and second moments, and further locations have less effect. 
In particular, $\sum_{j' \in I} \lambda^\sigma_{jj'} < 1$ for all $j \in J$ by assumption. 

We illustrate the effect of the above decision dependency in Figure \ref{momentFunctionsFigure}, where the first figure shows the change in the mean and the second figure depicts the change in the variance with respect to parameters $\lambda_j^\mu$ and $\lambda_j^\sigma$. For demonstration purposes, we assume the first open facility to be the closest one to customer site $j$, the second open facility to be the second closest, and so on. We highlight four different cases for these parameters such that in Case~(a), facility location decisions have no effect on demand distribution; in Case~(b) all facilities equally affect the first two moments; in Case~(c) closer facilities have higher impact; and in Case~(d) only the closest facility impacts customer demand. This illustration demonstrates different impacts of location decisions on customer demand, based on the dependency between moment information and customer behavior. 
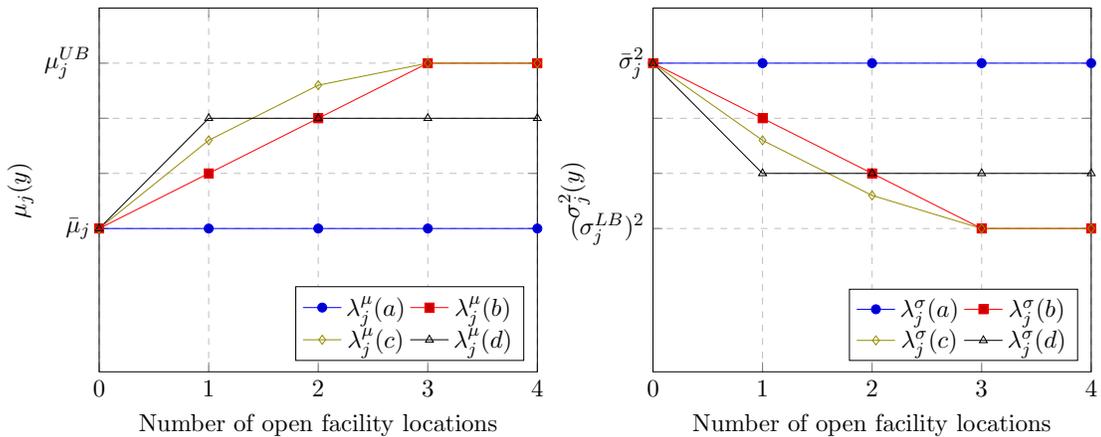
\begin{figure}[h]
\centering
\captionsetup{justification=centering}
\begin{subfigure}[b]{.45\textwidth}
  \centering
\begin{tikzpicture}[scale=0.85]
\begin{axis}[
	xlabel={Number of open facility locations},
	ylabel={$\mu_j (y)$},
	ymin =-3, ymax= 30,
	xmin = 0, xmax = 4,
	xtick = {0, 1, 2, 3, 4},
	ytick = {10,15,20,25},
	yticklabels = {$\bar{\mu}_j$, , ,$\mu_j^{UB}$},
     legend pos=south east,
    ymajorgrids=true,
    xmajorgrids=true,
    grid style=dashed,
    legend columns=2,
]

\addplot coordinates {
	(0,10)(1,10)(2,10)(3,10)(4,10)  
};

\addplot coordinates {
	(0,10)(1,15)(2,20)(3,25)(4,25)  
};

\addplot[mark=diamond,mark options={fill=white},color=olive]  coordinates {
	(0,10)(1,18)(2,23)(3,25)(4,25)  
};

\addplot[mark=triangle] coordinates {
	(0,10)(1,20)(2,20)(3,20)(4,20)  
};

\legend{$\lambda_j^\mu (a)$, $\lambda_j^\mu (b)$, $\lambda_j^\mu (c)$, $\lambda_j^\mu (d)$}
\end{axis}
\end{tikzpicture}
\end{subfigure}
\begin{subfigure}[b]{.45\textwidth}
  \centering
  \begin{tikzpicture}[scale=0.85]
\begin{axis}[
	xlabel={Number of open facility locations},
	ylabel={$\sigma_j^2(y)$},
	ymin = -3, ymax= 30,
	xmin = 0, xmax = 4,
	xtick = {0, 1, 2, 3, 4},
	ytick = {10,15,20,25},
	yticklabels = {$(\sigma_j^{LB})^{2}$, , ,$\bar{\sigma}_j^2$},
     legend pos=south east,
    ymajorgrids=true,
    xmajorgrids=true,
    grid style=dashed,
    legend columns=2,
]

\addplot coordinates {
	(0,25)(1,25)(2,25)(3,25)(4,25)  
};
\addplot coordinates {
	(0,25)(1,20)(2,15)(3,10)(4,10)  
};

\addplot[mark=diamond,mark options={fill=white},color=olive] coordinates {
	(0,25)(1,18)(2,13)(3,10)(4,10)  
};

\addplot[mark=triangle] coordinates {
	(0,25)(1,15)(2,15)(3,15)(4,15)  
};

\legend{$\lambda_j^\sigma (a)$, $\lambda_j^\sigma (b)$, $\lambda_j^\sigma (c)$, $\lambda_j^\sigma (d)$}
\end{axis}
\end{tikzpicture}
\end{subfigure}
\caption{Effect of the open facility locations on the moment information of demand.}
\label{momentFunctionsFigure}
\end{figure}

\subsection{DRO model and reformulation}
\label{optimizationModelSection}

In addition to decision variables $y_i, \ i \in I$, we define decision variables $x_{ij}$ and $s_j$ denoting at each customer site $j, \ \forall j \in J$, the amount of demand satisfied by facility $i$, and unsatisfied demand amount, respectively. Parameters $f_i$, $c_{ij}$, $p_j$, $r_j$ represent the cost of opening a facility at location $i$, unit transportation cost from location $i$ to site $j$, penalty of each unit of unsatisfied demand at site $j$, and revenue for satisfying each unit of demand at site $j$, respectively. We assume that the unit penalty of unmet demand at each customer site is higher than the unit cost of transportation from any two location pairs, i.e., $p_j > c_{ij}, \ \forall i \in I, \ j \in J$. This assumption is sensible in many business settings to assure the quality of service as high as possible, via guaranteeing customer satisfaction. 

Furthermore, instead of assuming a total amount of capacity at each individual facility, we consider a relaxed capacity restriction and assume that the capacity at each facility is pre-divided for individual customer sites. For example, to prepare for shipments, different sizes of vehicle fleets are pre-booked and scheduled to serve customers in different regions. Denote the total capacity in each location $i$ as $\sum_{j \in J} C_{ij}$, where $C_{ij}$ is the capacity at location $i$ dedicated to customer site $j$. For notational convenience, without loss of generality, we further simplify the case by assuming the same amount of capacity pre-allocated to serve each customer (i.e., $C_{ij}$ is the same and equals to $C_i$ for all the customer sites $j$).
   
We formulate the decision-dependent distributionally robust facility location problem as: 
\begin{equation} \label{eq:distRobustFormulation}
\min_{y \in \mathcal{Y} \subseteq \{0,1\}^{|I|}} \left\{ \sum_{i \in I } f_i y_i + \max_{\pi \in U(y)} \mathbf{E}_{\pi} [h(y,d(y))] \right \}, 
\end{equation}
\begin{subequations} \label{eq:distRobustFormulationMostInner}
\begin{alignat}{1}
\mbox{ where } \ h(y,d(y)) = \min_{x,s} \quad & \sum_{i \in I } \sum_{j \in J} c_{ij} x_{ij} + \sum_{j \in J} (p_j  s_j - r_j d_j(y)) \\ 
\text{s.t.} \quad & \sum_{i \in I} x_{ij} + s_j = d_j(y) \quad \forall j \in J \label{eq:distRobustFormulationDemandConstr} \\
& x_{ij} \leq C_i y_i \quad \forall i \in I, j \in J \label{eq:distRobustFormulationCapacityConstr} \\
& s_i, x_{ij} \geq 0 \quad \forall i \in I, j \in J.
\end{alignat}
\end{subequations}
The objective function \eqref{eq:distRobustFormulation} minimizes the total cost of locating facilities and the maximum expected cost of transportation and unmet demand minus revenue for any demand distribution $\pi \in U(y)$. We let the polyhedron $\mathcal{Y}$ include constraints that are solely related to facility-location decisions. Constraint \eqref{eq:distRobustFormulationDemandConstr} ensures that demand at each customer site is either satisfied by other locations or penalized, while constraint \eqref{eq:distRobustFormulationCapacityConstr} enforces capacity restriction for each open facility $i \in I$.

To derive  a single-level reformulation that can be directly handled by off-the-shelf solvers, we first show a closed-form solution to the inner problem \eqref{eq:distRobustFormulationMostInner}. 

\begin{prop} \label{propInnerProbReformulation}
The optimal objective value of problem \eqref{eq:distRobustFormulationMostInner} can be computed by
\begin{equation} \label{eq:closedFormInnerH}
h(y,d(y)) = \sum_{j \in J} \left( \max_{i^* = 0, 1, \cdots, |I|} \left\{ c_{i^*j} d_j(y) + \sum_{i \in I: c_{ij} < c_{i^*j}} C_i y_i (c_{ij} - c_{i^*j}) \right\} - r_j d_j(y) \right),
\end{equation}
where $c_{0j} := p_j$.
\end{prop}
\begin{proof}
Note that the most inner problem \eqref{eq:distRobustFormulationMostInner} can be decomposed with respect to each location $j$. Therefore, we express $h(y,d(y))$ as $\sum_{j \in I} h_j(y,d(y))$, where
\begin{subequations} \label{eq:decomposedProgram}
\begin{alignat}{1}
h_j(y,d(y)) = \min_{x_{.j}, s_j} \quad &  \sum_{i \in I } c_{ij} x_{ij} + p_j  s_j  - r_j d_j(y)\\
\text{s.t.} \quad & \sum_{i \in I} x_{ij} + s_j = d_j(y) \label{eq:decomposedConstr1}\\
& x_{ij} \leq C_i y_i \quad \forall i \in I \label{eq:decomposedConstr2}\\\
& s_j, x_{ij} \geq 0 \quad \forall i \in I.
\end{alignat}
\end{subequations}

Let $\beta$ and $\upsilon_i$ be the dual variables associated with constraints \eqref{eq:decomposedConstr1} and \eqref{eq:decomposedConstr2}, respectively. We formulate the dual of model \eqref{eq:decomposedProgram} as 
\begin{subequations} \label{eq:decomposedProgramDual}
\begin{alignat}{1}
\max_{\beta, \> \upsilon_i} \quad & \beta d_j(y) + \sum_{i \in I } C_i y_i \upsilon_i \\
\text{s.t.} \quad & \beta + \upsilon_i \leq c_{ij} \quad \forall i \in I \label{eq:decomposedDualConstr1} \\
& \beta \leq p_j \\
& \upsilon_i \leq 0 \quad \forall i \in I
\end{alignat}
\end{subequations}
To identify the optimal objective value of model \eqref{eq:decomposedProgramDual}, we derive the extreme points of its feasible region. 
To this end, we examine two cases through counting the number of tight constraints. 
\begin{enumerate}
\item \underline{$ \beta = p_j$ :} In this case,  for all $i \in I$, either $\upsilon_i = 0$ or $\upsilon_i = c_{ij} - p_j$. Due to \eqref{eq:decomposedDualConstr1} and $p_j > c_{ij}$, we have $\upsilon_i \leq c_{ij} - p_j < 0$, making the condition $\upsilon_i = 0$ redundant. Therefore, when $ \beta = p_j$, $\upsilon_i$ equals to $c_{ij} - p_j$ in the corresponding extreme point. The value of the objective function then becomes $p_j d_j(y) + \sum_{i \in I} C_i y_i (c_{ij} - p_j)$.

\item \underline{$ \beta < p_j$ :} In this case, for all $i \in I$, either $\upsilon_i = 0$ or $\upsilon_i = c_{ij} - \beta$. Additionally, there exists at least one location $i^*$ such that $\upsilon_{i^*} = c_{i^*j} - \beta = 0$. Therefore, at least $|I| + 1$ constraints are satisfied at an extreme point. Thus, $\beta = c_{i^*j}$ for some $i^* \in I$. For $i \in I \setminus \{ i ^* \}$, we have either $\upsilon_i = 0$ or $\upsilon_i = c_{ij} - c_{i^*j}$. Since $\upsilon_i \leq c_{ij} - c_{i^*j}$ and $\upsilon_i \leq 0$, if $c_{ij} < c_{i^*j}$, then $\upsilon_i = c_{ij} - c_{i^*j}$. Otherwise, $\upsilon_i = 0$ because we maximize a positive number times $\upsilon_i$ in the objective. For a given $i^*$ location, the objective function becomes $c_{i^*j} d_j(y) + \sum_{i \in I: c_{ij} < c_{i^*j}} C_i y_i (c_{ij} - c_{i^*j})$.
\end{enumerate}

Combining the above two cases, we obtain a closed-form expression for the optimal objective value of model \eqref{eq:decomposedProgramDual}. Since $ p_j = c_{0j} > c_{ij}, \ \forall i \in I$, the optimal objective value of the problem can be expressed as
\begin{equation} \label{eq:closedFormInnerHDecomposed}
\max_{i^* = 0, 1, \cdots, |I|} \left\{ c_{i^*j} d_j(y) + \sum_{i \in I: c_{ij} < c_{i^*j}} C_i y_i (c_{ij} - c_{i^*j})   \right\}.
\end{equation}
As the program \eqref{eq:decomposedProgramDual} is feasible and bounded, strong duality holds between models \eqref{eq:decomposedProgram} and \eqref{eq:decomposedProgramDual}. As a result, the optimal objective value of \eqref{eq:decomposedProgram} equals to 
\begin{equation} \label{eq:closedFormInnerHDecomposed2}
\max_{i^* = 0, 1, \cdots, |I|} \left\{ c_{i^*j} d_j(y) + \sum_{i \in I: c_{ij} < c_{i^*j}} C_i y_i (c_{ij} - c_{i^*j}) \right\} - r_j d_j(y),
\end{equation}
which completes the proof.
\end{proof}

\begin{thm} \label{reformulationTheorem}
Problem \eqref{eq:distRobustFormulation} can be reformulated as follows: 
\begin{subequations} \label{eq:allFormulationv1}
\begin{alignat}{1}
\min_{y, \alpha, \delta^1, \delta^2, \gamma^1, \gamma^2} \quad & f^\top y + \sum_{j \in J} \bigg( \alpha_j + \delta_j^1 (\mu_j(y) + \epsilon^\mu_j) - \delta_j^2 (\mu_j(y) - \epsilon^\mu_j) \notag \\
& \qquad \qquad \qquad + \gamma_j^1 (\sigma_j^2(y) + (\mu_j(y))^2) \overline{\epsilon}^\sigma_j- \gamma_j^2 (\sigma_j^2(y) + (\mu_j(y))^2) \underline{\epsilon}^\sigma_j\bigg) \label{eq:objNonlinearFull} \\
\text{s.t.} \quad & \alpha_j + (\delta_j^1 - \delta_j^2) d_k + (\gamma_j^1 - \gamma_j^2) d_k^2 \geq \theta_{jk}(y) \quad \forall j \in J, k = 1, \cdots, K, \label{eq:DualConstraint1before}\\
& y \in \mathcal{Y} \subseteq \{0, 1\}^{|I|}, \delta_j^1, \gamma_j^1, \delta_j^2, \gamma_j^2 \geq 0 \quad \forall j \in J,
\end{alignat}
\end{subequations}
where $\theta_{jk}(y) = c_{i_{jk}^* j} d_k + \sum_{i \in I: c_{ij} < c_{i_{jk}^* j}} C_i y_i (c_{ij} - c_{i_{jk}^* j}) - r_j d_k$ and $i_{jk}^*$ is the maximizer of expression~\eqref{eq:closedFormInnerHDecomposed2} with $d_j(y)$ being replaced by $d_k$.
\end{thm}
\begin{proof}
Following Proposition \ref{propInnerProbReformulation}, we can reformulate the inner problem $\max_{\pi \in U(y)} \mathbf{E} [h(y,d(y))]$ for a given $y$ as
\begin{subequations} \label{eq:innerProblem}
\begin{alignat}{1}
\max_{\pi_{jk}, j \in J, k = 1, \cdots, K} \quad &  \sum_{j \in J} \sum_{k = 1}^K \pi_{jk} \left( (c_{i_{jk}^* j} - r_j) d_k + \sum_{i \in I: c_{ij} < c_{i_{jk}^* j}} C_i y_i (c_{ij} - c_{i_{jk}^* j})  \right) \\
\text{s.t.} \quad &  \sum_{k =1}^K \pi_{jk} = 1 \quad \forall j \in J, \\
& \sum_{k =1}^K \pi_{jk} d_k \leq \mu_j(y) + \epsilon^\mu_j \quad \forall j \in J, \\
& \sum_{k =1}^K \pi_{jk} d_k \geq \mu_j(y) - \epsilon^\mu_j \quad \forall j \in J, \\
& \sum_{k =1}^K \pi_{jk} d_k^2 \leq (\sigma_j^2(y) + (\mu_j(y))^2) \overline{\epsilon}^\sigma_j\quad \forall j \in J, \\
& \sum_{k =1}^K \pi_{jk} d_k^2 \geq (\sigma_j^2(y) + (\mu_j(y))^2) \underline{\epsilon}^\sigma_j\quad \forall j \in J, \\
& \pi_{jk} \geq 0 \quad \forall j \in J, k = 1, \cdots, K.
\end{alignat}
\end{subequations}

Let $\alpha_j, \ \delta_j^1, \ \delta_j^2, \ \gamma_j^1, \ \gamma_j^2 $ for all $j \in J$ be the dual variables associated with all the constraints in model \eqref{eq:innerProblem}. Then, we can formulate the corresponding dual of model \eqref{eq:innerProblem} as
\begin{subequations} \label{eq:innerDualProblem}
\begin{alignat}{1}
\min_{\alpha, \delta^1, \delta^2, \gamma^1, \gamma^2} \quad & \sum_{j \in J} \bigg( \alpha_j + \delta_j^1 (\mu_j(y) + \epsilon^\mu_j) - \delta_j^2 (\mu_j(y) - \epsilon^\mu_j) \notag \\
& \qquad \qquad \qquad + \gamma_j^1 (\sigma_j^2(y) + (\mu_j(y))^2) \overline{\epsilon}^\sigma_j- \gamma_j^2 (\sigma_j^2(y) + (\mu_j(y))^2) \underline{\epsilon}^\sigma_j\bigg) \\
\text{s.t.} \quad & \alpha_j + (\delta_j^1 - \delta_j^2) d_k + (\gamma_j^1 - \gamma_j^2) d_k^2 \geq \theta_{jk}(y) \quad \forall j \in J, k = 1, \cdots, K, \\
& \delta_j^1, \gamma_j^1, \delta_j^2, \gamma_j^2 \geq 0 \quad \forall j \in J.
\end{alignat}
\end{subequations}
As a result, we can express model \eqref{eq:distRobustFormulation} in the form of \eqref{eq:allFormulationv1}. This completes the proof. 
\end{proof}

Model~\eqref{eq:allFormulationv1} is a mixed-integer nonlinear program due to the nonlinear objective function \eqref{eq:objNonlinearFull}. To linearize it, we assume upper bounds $\overline{\delta^1},  \overline{\delta^2},  \overline{\gamma^1},  \overline{\gamma^2}$ on the variables $\delta^1, \delta^2, \gamma^1, \gamma^2$, respectively. Using these bounds, McCormick envelopes can be applied for linearizing the bilinear terms in the objective function \eqref{eq:objNonlinearFull} \citep{McCormick1976}. Specifically, we define set $M'_{(\underline{\eta}, \overline{\eta})}$ involving the McCormick inequalities for linearizing any bilinear term $w' = \eta z$ when $\eta \in [\underline{\eta}, \overline{\eta}]$ and $z \in \{0, 1\}$  and give the details as follows.
\begin{equation} \label{eq:MccormickSetBilinear}
M'_{(\underline{\eta}, \overline{\eta})} = \Bigl\{(w', \eta, z) \in \mathbb{R}^3: \eta - (1 - z) \overline{\eta} \leq w' \leq \eta - \underline{\eta} (1 - z), \underline{\eta} z \leq w' \leq \overline{\eta} z\Bigr\}.
\end{equation}
Because variable $z$ is binary valued, we have an exact reformulation in \eqref{eq:MccormickSetBilinear} for representing the bilinear terms. Similarly, we define set $M''_{(\underline{\eta}, \overline{\eta})}$ involving McCormick inequalities for linearizing any trilinear term $w'' = \eta z_1 z_2$ when $\eta \in [\underline{\eta}, \overline{\eta}]$ such that $\underline{\eta} \geq 0$, and $z_1, z_2 \in \{0, 1\}$ as follows. 
\begin{align} \label{eq:MccormickSetTrilinear}
M''_{(\underline{\eta}, \overline{\eta})} = \Bigl\{(w'',  \eta, & z_1, z_2 ) \in  \mathbb{R}^4: w'' \leq \overline{\eta} z_1, w'' \leq \overline{\eta} z_2, w'' \leq \eta - \underline{\eta}(1 - z_1), w'' \leq \eta - \underline{\eta}(1 - z_2), \notag \\
& w'' \geq \underline{\eta}(-1 + z_1 + z_2), w'' \geq \eta + \overline{\eta}(-2 + z_1 + z_2), z_1 \leq 1, z_2 \leq 1, \underline{\eta} \leq \eta \leq \overline{\eta} \Bigr\}.
\end{align}

The discussed trilinear case \eqref{eq:MccormickSetTrilinear} involves two binary variables, and based on existing results, we confirm that it provides an exact reformulation in the following proposition. 
\begin{prop} \citep{Meyer2004} Let $0 \leq \underline{\eta} \leq \overline{\eta}$. Then  $M''_{(\underline{\eta}, \overline{\eta})} = conv\Bigl(\Bigl\{(w, \eta, z_1, z_2): w = \eta z_1 z_2, \eta \in [\underline{\eta}, \overline{\eta}], z_1, z_2 \in \{0, 1\} \Bigr\}\Bigr)$.
\end{prop}

Motivated by carsharing application, we study the case when the market capacity is sufficiently large such that the mean of the random demand at each customer site $j$ is not restricted by $\mu_j^{UB}$ in Equation~\eqref{eq:momentFunctions}. Similarly, we omit the lower bound restriction $(\sigma_j^{LB})^{2}$ for the second-moment information. However, these assumptions are not restrictive in terms of the complexity of the problem formulation. In the presence of these upper and lower bounds, we can model the moment functions~\eqref{eq:momentFunctions} as piecewise linear functions with additional binary variables. The arising nonlinear relationships can be further linearized using McCormick envelopes. 

Following the above assumptions, the ambiguity set $U(y)$ in \eqref{eq:ambiguitySet} contains  nonlinear terms in $y$ if using mean and standard deviation functions defined in \eqref{eq:momentFunctions}. Specifically,  
\begin{subequations}
\begin{align}
(\mu_j(y))^2 & = \bar{\mu}^2_j \left(1 + 2 \sum_{j' \in I} \lambda^\mu_{jj'} y_{j'} + \sum_{j' \in I} (\lambda^\mu_{jj'})^2 y^2_{j'} + 2 \sum_{l = 1}^{|I|} \sum_{m = 1}^{l - 1} \lambda^\mu_{jl} \lambda^\mu_{jm} y_l y_m \right) \\
& = \bar{\mu}^2_j \left(1 + \sum_{j' \in I} ( 2 \lambda^\mu_{jj'} + (\lambda^\mu_{jj'})^2) y_{j'} + 2 \sum_{l = 1}^{|I|} \sum_{m = 1}^{l - 1} \lambda^\mu_{jl} 
\lambda^\mu_{jm} y_l y_m \right). \label{eq:muSquareDefinition}
\end{align}
\end{subequations}
To linearize the above expression, define a new variable $Y_{lm} := y_l y_m$ where $(Y_{lm}, y_l, y_m) \in M'_{(0,1)}$. To linearize the nonlinear terms in the objective function \eqref{eq:objNonlinearFull}, let $\Delta^h_{jj'} := \delta^h_{j} y_{j'}$, $\Gamma^h_{jj'} := \gamma^h_{j} y_{j'}$, $\Psi^h_{jlm} := \gamma^h_{j} y_{l} y_{m}$, for $h = 1, \ 2$. For any pair of $j \in J$ and  $j' \in I$, denote $\Lambda_{jj'} :=  -\bar{\sigma}_j^2 \lambda^\sigma_{jj'} + \bar{\mu}^2_j (2\lambda^\mu_{jj'} + (\lambda^\mu_{jj'})^2)$ as the parameters specific to the values of $\lambda^\mu$, $\lambda^\sigma$, as well as empirical moment estimates for any pair $j \in J, \ j' \in I$. Combining the above result with Theorem \ref{reformulationTheorem}, we derive a mixed-integer linear programming reformulation \eqref{eq:allFormulationv2} of model \eqref{eq:allFormulationv1} under ambiguity set~\eqref{eq:ambiguitySet} in the following theorem. 
\begin{thm}
Problem \eqref{eq:distRobustFormulation} is equivalent to the following mixed-integer linear program  \eqref{eq:allFormulationv2}. 
\begin{subequations} \label{eq:allFormulationv2}
\begin{alignat}{1}
\min \quad & f^\top y + \sum_{j \in J} \bigg( \alpha_j + \delta_j^1 (\bar{\mu}_j + \epsilon^\mu_j) - \delta_j^2 (\bar{\mu}_j - \epsilon^\mu_j) + \bar{\mu}_j \sum_{j' \in I} \lambda^\mu_{jj'}( \Delta^1_{jj'} - \Delta^2_{jj'}) \notag \\
& + (\bar{\sigma}_j^2 + \bar{\mu}^2_j) (\overline{\epsilon}^\sigma_j\gamma^1_j  - \underline{\epsilon}^\sigma_j \gamma^2_j) +  \sum_{j' \in I}  \Lambda_{jj'}  (\overline{\epsilon}^\sigma_j\Gamma^1_{jj'} - \underline{\epsilon}^\sigma_j \Gamma^2_{jj'}) + 2 \bar{\mu}^2_j \sum_{l = 1}^{|I|} \sum_{m = 1}^{l - 1} \lambda^\mu_{jl} \lambda^\mu_{jm} (\overline{\epsilon}^\sigma_j\Psi^1_{jlm} - \underline{\epsilon}^\sigma_j\Psi^2_{jlm}) \bigg) \label{eq:finalModelObj} \\
\text{s.t.} \quad & \alpha_j + (\delta_j^1 -\delta_j^2) d_k + (\gamma_j^1 - \gamma_j^2) d_k^2 \geq (c_{i^*j} - r_j) d_k + \sum_{i \in I: c_{ij} < c_{i^*j}} C_i y_i (c_{ij} - c_{i^*j}) \notag \\
& \hspace{6.5cm} \forall i^* \in I \cup \{0\}, j \in I, k = 1, \cdots, K \label{eq:constraintNonlinearFull} \\
& (\Delta^h_{jj'}, \delta^h_{j}, y_{j'}) \in M'_{(0,\overline{\delta_j^h})}, (\Gamma^h_{jj'}, \gamma^h_{j},  y_{j'}) \in M'_{(0,\overline{\gamma_j^h})} \quad \forall j \in J, j' \in I, h = 1,2 \\
& (\Psi^h_{jlm}, \gamma^h_{j}, y_{l}, y_{m}) \in M''_{(0,\overline{\gamma_j^h})} \quad \forall j \in J, l = 1, \dots, |I|, l > m \\
& y \in \mathcal{Y} \subseteq \{0, 1\}^{|I|}, \delta_j^1, \gamma_j^1, \delta_j^2, \gamma_j^2 \geq 0 \quad \forall j \in J.
\end{alignat}
\end{subequations}
\end{thm}

\begin{proof}
We linearize formulation \eqref{eq:allFormulationv1} obtained in Theorem \ref{reformulationTheorem} to derive a mixed-integer linear programming reformulation. To this end, we first plug in the the decision-dependent moment information at each site $j$, $\mu_j(y)$ and $\sigma_j^2(y)$, into the objective function \eqref{eq:objNonlinearFull}, using definitions in \eqref{eq:momentFunctions} and \eqref{eq:muSquareDefinition}. As the resulting objective function includes bilinear and trilinear terms, we introduce new variables to obtain the linear objective function \eqref{eq:finalModelObj}. Constraint \eqref{eq:constraintNonlinearFull} corresponds to  \eqref{eq:DualConstraint1before}, which is also linearized. 
The remaining constraints refer to definitions of the newly introduced variables, their corresponding McCormick constraints in the forms of \eqref{eq:MccormickSetBilinear} and \eqref{eq:MccormickSetTrilinear}, restrictions on the facility-location variable $y$, and  the non-negativity constraints on all the decision variables. 
\end{proof}

\subsection{Valid inequalities}
\label{validInequalitiesSection}

Next, we examine the underlying problem structure for deriving valid inequalities to obtain a stronger formulation of the mixed-integer linear programming reformulation  \eqref{eq:allFormulationv2}. We first present an intermediate result using the inner problem \eqref{eq:innerProblem}. Since the dual \eqref{eq:innerDualProblem} of the inner problem is decomposable with respect to each location $j$, we study the following decomposed formulation for every $j \in J$.
\begin{subequations} \label{eq:innerDualProblemDecomposed}
\begin{alignat}{1}
\min_{\alpha_j, \delta_j^1, \delta_j^2, \gamma_j^1, \gamma_j^2} \quad & \alpha_j + \delta_j^1 (\mu_j(y) + \epsilon^\mu_j) - \delta_j^2 (\mu_j(y) - \epsilon^\mu_j) + \gamma_j^1 (\sigma_j^2(y) + (\mu_j(y))^2) \overline{\epsilon}^\sigma_j\notag \\
& \qquad \qquad \qquad \qquad \qquad \qquad \qquad \qquad - \gamma_j^2 (\sigma_j^2(y) + (\mu_j(y))^2) \underline{\epsilon}^\sigma_j \label{eq:innerDualProblemDecomposedObj} \\
\text{s.t.} \quad & \alpha_j + (\delta_j^1 - \delta_j^2) d_k + (\gamma_j^1 - \gamma_j^2) d_k^2 \geq \theta_{jk}(y) \quad k = 1, \cdots, K, \label{eq:DualConstraint1} \\
& \delta_j^1, \gamma_j^1, \delta_j^2, \gamma_j^2 \geq 0. \label{eq:DualConstraint2}
\end{alignat}
\end{subequations}

\begin{lem} \label{extremeRayProp}
Extreme rays of the feasible set $\{(\alpha_j, \delta_j^1, \delta_j^2, \gamma_j^1, \gamma_j^2): \eqref{eq:DualConstraint1}, \eqref{eq:DualConstraint2}\}$ are
\begin{enumerate}
\item $(d_{(1)} d_{(2)}, 0, d_{(1)} + d_{(2)}, 1, 0)$
\item $(d_{(K-1)} d_{(K)}, 0, d_{(K-1)} + d_{(K)}, 1, 0)$
\item $(- d_{(1)} d_{(K)}, d_{(1)} + d_{(K)}, 0, 0, 1)$
\end{enumerate}
where $d_{(1)}, \cdots, d_{(K)}$ represent the ordered sequence of the support of the random demand. 
\end{lem}

\begin{proof}
Since $\delta_j := \delta_j^1 - \delta_j^2$ and $\gamma_j := \gamma_j^1 - \gamma_j^2$ are unbounded, we can equivalently consider the following system of inequalities in place of \eqref{eq:DualConstraint1} and  \eqref{eq:DualConstraint2}
\begin{equation} \label{eq:innerDualProblemDecomposedSmaller}
\alpha_j + \delta_j d_k + \gamma_j d_k^2 \geq \theta_{jk}(y) \quad k = 1, \cdots, K. 
\end{equation}
To identify extreme rays, we solve the inequality system \eqref{eq:extremeRaySystem} for $m, \ n \in \{1, \cdots, K\}$;
\begin{subequations} \label{eq:extremeRaySystem}
\begin{alignat}{1}
\alpha_j + \delta_j d_m + \gamma_j d_m^2 & = 0 \label{eq:extremeRayEq1} \\
\alpha_j + \delta_j d_n + \gamma_j d_n^2 & = 0 \label{eq:extremeRayEq2} \\
\alpha_j + \delta_j d_k + \gamma_j d_k^2 & \geq 0 \quad k \in \{1, \cdots, K\} \setminus \{ m, n\}. \label{eq:extremeRayIneq}
\end{alignat}
\end{subequations}
Without loss of generality, we assume that $d_m < d_n$. Solving the equalities \eqref{eq:extremeRayEq1} and \eqref{eq:extremeRayEq2}, we obtain $\delta_j = - (d_m + d_n) \gamma_j$, and $\alpha_j = d_m d_n \gamma_j$. The next step is to ensure that the inequality system \eqref{eq:extremeRayIneq} is satisfied. We study two cases with respect to the direction  $\gamma_j$ as follows by normalizing $|\gamma_j| = 1$.
\begin{enumerate}
\item \underline{$\gamma_j = 1$}: In this case, we need to guarantee that $(d_k - d_m) (d_k - d_n) \geq 0$ for all $k \in \{1, \cdots, K\} \setminus \{ m, n\}$. Consequently, we have either $d_k \geq d_m$ and $d_k \geq d_n$, or $d_k \leq d_m$ and $d_k \leq d_n$. There are only two ways to satisfy these restrictions, resulting in the following extreme ray generators of the form $(\alpha_j, \delta_j, \gamma_j)$:
\begin{itemize}
\item $(d_{(1)} d_{(2)}, - (d_{(1)} + d_{(2)}), 1)$;
\item $(d_{(K-1)} d_{(K)}, - (d_{(K-1)} + d_{(K)}), 1)$.
\end{itemize}
\item \underline{$\gamma_j = -1$}: In this case, we need to ensure that $(d_k - d_m) (d_k - d_n) \leq 0$ for all $k \in \{1, \cdots, K\} \setminus \{ m, n\}$. This requires that $d_m \leq d_k \leq d_n$. To satisfy this case, we have the extreme ray generator
\begin{itemize}
\item $(-d_{(1)} d_{(K)}, d_{(1)} + d_{(K)}, -1)$.
\end{itemize}
\end{enumerate}
Lastly, through converting the resulting extreme ray generators to the original variables of the form $(\alpha_j, \delta_j^1, \delta_j^2, \gamma_j^1, \gamma_j^2)$ using $\delta_j^1 = \max\{0,\delta_j\}$, $\delta_j^2 = \max\{0,-\delta_j\}$,  $\gamma_j^1 = \max\{0,\gamma_j\}$, $\gamma_j^2 = \max\{0,-\gamma_j\}$, we obtain the desired result. This completes the proof. 
\end{proof}

Building on Proposition \ref{extremeRayProp}, we derive valid inequalities for the original problem \eqref{eq:distRobustFormulation} as follows.

\begin{prop} \label{propValidIneq}
The following inequalities are valid for problem \eqref{eq:distRobustFormulation}:
\begin{subequations} \label{eq:validInequalities}
\begin{alignat}{2}
&d_{(1)} d_{(2)} - (d_{(1)} + d_{(2)}) (\mu_j(y) - \epsilon^\mu_j) + (\sigma_j^2(y) + (\mu_j(y))^2) \overline{\epsilon}^\sigma_j \geq 0 \quad & \forall j \in  J \\
&d_{(K-1)} d_{(K)} - (d_{(K-1)} + d_{(K)}) (\mu_j(y) - \epsilon^\mu_j) + (\sigma_j^2(y) + (\mu_j(y))^2) \overline{\epsilon}^\sigma_j \geq 0 \quad & \forall j \in J \\ 
&- d_{(1)} d_{(K)} + (d_{(1)} + d_{(K)}) (\mu_j(y) + \epsilon^\mu_j) - (\sigma_j^2(y) + (\mu_j(y))^2) \underline{\epsilon}^\sigma_j \geq 0 \quad & \forall j \in J
\end{alignat}
\end{subequations}
\end{prop}

\begin{proof}
First, consider the primal problem~\eqref{eq:innerProblem} and its dual problem~\eqref{eq:innerDualProblem}. Note that the dual model~\eqref{eq:innerDualProblem} is always feasible as we can let values of variables $\alpha_j$ be arbitrarily large. To ensure the feasibility of the primal problem, it suffices to demonstrate that the dual problem is bounded. To this end, we consider the decomposed dual subproblem~\eqref{eq:innerDualProblemDecomposed}, and use the extreme ray generators in Lemma~\ref{extremeRayProp} by plugging them into the objective function~\eqref{eq:innerDualProblemDecomposedObj}. The resulting inequalities ~\eqref{eq:validInequalities} ensure the boundedness of the dual problem~\eqref{eq:innerDualProblemDecomposed} to guarantee the feasibility of \eqref{eq:innerProblem}. This completes the proof. 
\end{proof}
 
We continue to linearize nonlinear terms in \eqref{eq:validInequalities} using Equation \eqref{eq:muSquareDefinition} and McCormick envelopes \eqref{eq:MccormickSetBilinear}. As a result, inequalities \eqref{eq:validInequalities} are equivalent to:
\begin{subequations} \label{linearValidIneqs}
\begin{alignat}{1}
& d_{(1)} d_{(2)} - (d_{(1)} + d_{(2)}) (\bar{\mu}_j (1 + \sum_{j' \in I} \lambda^\mu_{jj'} y_{j'}) - \epsilon^\mu_j) + \Theta_j \overline{\epsilon}^\sigma_j\geq 0 \quad \forall j \in J \label{eq:validConstr1} \\
& d_{(K-1)} d_{(K)} - (d_{(K-1)} + d_{(K)}) (\bar{\mu}_j (1 + \sum_{j' \in I} \lambda^\mu_{jj'} y_{j'}) - \epsilon^\mu_j) + \Theta_j \overline{\epsilon}^\sigma_j\geq 0 \quad \forall j \in J \label{eq:validConstr2} \\ 
& - d_{(1)} d_{(K)} + (d_{(1)} + d_{(K)}) (\bar{\mu}_j (1 + \sum_{j' \in I} \lambda^\mu_{jj'} y_{j'}) + \epsilon^\mu_j) - \Theta_j \underline{\epsilon}^\sigma_j\geq 0 \quad \forall j \in J \label{eq:validConstr3} \\
& \Theta_j = \bar{\sigma}_j^2 + \bar{\mu}^2_j + \sum_{j' \in I} \Lambda_{jj'} y_{j'} + 2 \bar{\mu}^2_j\sum_{l = 1}^{|I|} \sum_{m = 1}^{l - 1} \lambda^\mu_{jl} \lambda^\mu_{jm} Y_{lm} \quad \forall j \in J \label{eq:validConstr4}\\
& (Y_{lm}, y_l, y_m) \in M'_{(0,1)} \quad \forall l = 1, \dots, |I|, l > m \label{eq:validConstr5}
\end{alignat}
\end{subequations}
After integrating constraints~\eqref{linearValidIneqs} into the model~\eqref{eq:allFormulationv2}, we strengthen our formulation for the original problem~\eqref{eq:distRobustFormulation}. Later our computational studies are based on the formulation~\eqref{eq:allFormulationv2} with valid inequalities~\eqref{linearValidIneqs}, and we further demonstrate the efficiency of the proposed constraints in the next section. 

\section{Computational Studies}
\label{Computations}

We demonstrate the efficacy of the proposed decision-dependent distributionally robust (DDDR) approach from various aspects. We also compare its solutions and performance against facility location plans obtained from distributionally robust (DR) and stochastic programming (SP) approaches neglecting decision-dependency. 

To evaluate a location plan $\hat{y}$, we run out-of-sample test by using a benchmark stochastic programming model \eqref{eq:SPFormulation} in assessing potential solution performance. The Monte Carlo sampling approach and Sample Average Approximation method \citep[see][]{Kleywegt2002} are adopted for generating realizations of the underlying uncertainty in customer demand. Specifically, we consider a given set of demand scenarios $d^{\omega}_j(\hat{y})$ for all $\omega \in \Omega$, where the scenarios are generated based on plan $\hat{y}$ using the moment information defined in~\eqref{eq:momentFunctions}. For each scenario $\omega$, let $p^{\omega}$, $x^{\omega}_{ij}$ and $s^{\omega}_j$ be the probability of realizing the scenario, the amount of demand at customer site $j$ satisfied by facility at location $i$, and the unsatisfied demand at customer site $j$, respectively. A solution evaluation model is: 
\begin{subequations} 
\label{eq:SPFormulation}
\begin{alignat}{1}
\min_{x,s} \quad & \sum_{i \in I} f_i \hat{y}_i + \sum_{\omega \in \Omega} p^{\omega} \left( \sum_{i \in I} \sum_{j \in J} c_{ij} x^{\omega}_{ij} + \sum_{j \in J} \left( p_j  s^{\omega}_j - r_j d^{\omega}_j(\hat{y}) \right) \right) \label{eq:SPFormulationObj}\\ 
\text{s.t.} \quad & \sum_{i \in I} x^{\omega}_{ij} + s^{\omega}_j = d^{\omega}_j(\hat{y}) \quad \forall j \in J, \omega \in \Omega \label{eq:SPFormulationDemandConstr} \\
& x_{ij}^{\omega} \leq C_i \hat{y}_i \quad \forall i \in I, j \in J, \omega \in \Omega \label{eq:SPFormulationCapacityConstr} \\
& s^{\omega}_i, x^{\omega}_{ij} \geq 0 \quad \forall i \in I, j \in J, \omega \in \Omega.
\end{alignat}
\end{subequations}
The objective function \eqref{eq:SPFormulationObj} minimizes the total expected cost of facility location, transportation, and unmet demand minus the revenue obtained. Constraint \eqref{eq:SPFormulationDemandConstr} ensures that demand is either satisfied or penalized across all the scenarios while constraint \eqref{eq:SPFormulationCapacityConstr} guarantees that the capacity of each facility location is not violated. For an independently and identically distributed set of scenarios, model \eqref{eq:SPFormulation} is decomposable by scenario when the value $\hat{y}$ of the first-stage decision vector $y$ is given. In this case, model \eqref{eq:SPFormulation} can be solved separately for each scenario subproblem. 

In the remainder of the section, we first discuss experimental settings used in our numerical studies in Section~\ref{expSettingsSection}. Then we provide a comprehensive analysis of the proposed approach on various test cases in Section~\ref{caseStudySubsection}  including different (i) variability levels of demand, (ii) unit penalty costs, (iii) robustness levels, (iv) limits on the number of open facilities, and (v) decision-dependent distribution models. Finally, we highlight the computational efficiency of the DDDR model by conducting experiments using different sizes of instances in Section \ref{compPerfSection}. 

\subsection{Experimental Setup}
\label{expSettingsSection}

We randomly generate a set of potential facility locations and customer sites. We first present the default settings for all the problem parameters, which remain the same  throughout all the numerical studies, unless otherwise stated. Euclidean distance is used to represent the distance between each candidate facility location and customer site. These distance values are assumed to directly affect transportation cost parameters, namely $c_{ij}$ for all $i \in I$, $j \in J$. The parameters for the fixed opening cost, $f_i$, and capacity, $C_i$, for all $i \in I$ are sampled from Uniform distributions $U(5000, 10000)$ and $U(10,20)$, respectively. Furthermore, for each $j \in J$, we set unit penalty, $p_j$, for the unmet demand as 225, and revenue parameter $r_j$ as 150 for each customer site $j \in J$.
 
We sample the empirical mean of demand at each customer site $j \in J$, ${\bar{\mu}_j}$, from a Uniform distribution $U(20, 40)$. Then, we let $\bar{\sigma}_j = {\bar{\mu}_j}$, implying the coefficient of variation equaling to 1. 
We define the moment-based ambiguity set by letting $\epsilon_j^\mu = 0$, and $\underline{\epsilon}^\sigma_j$ = $\overline{\epsilon}^\sigma_j$ = 1 in \eqref{eq:ambiguitySet}, for all customer sites $j \in J$. The support size of demand values at each customer site, namely $K$, is taken as 100 and thus the values $d_1, \cdots, d_K$ are in the range $\{1, \cdots, 100\}$. 

For establishing decision dependency between demand distribution and facility location decisions, we select the parameters $\lambda_{ji}^\mu$ and $\lambda_{ji}^\sigma$ for all $i \in I$, $j \in J$ using the distance between each facility location and customer pair. We consider them as a decreasing function of the corresponding distance, specifically $\exp(-c_{ij}/25)$. Consequently, the effect of a facility located at $i$ on the demand at customer site $j$ is higher when the facility is closer to the customer. Next, the sums of the vectors $\lambda_j^\mu$, $\lambda_j^\sigma$ are normalized for each customer site $j \in J$ to adjust the effect of the location decisions on demand. Note that if $\lambda^\mu_{j i}$ and $\lambda^\sigma_{j i}$ values are set to 0 for all $j \in J$ and $i \in I$ in the moment functions in \eqref{eq:momentFunctions}, then the current setting reduces to a decision-independent form, i.e., a traditional distributionally robust optimization model. 

To assess the performance of the proposed optimization framework by taking into account various choices of model parameters and underlying demand distribution, we provide an extensive set of numerical studies over the proposed and existing optimization approaches. We implement all the optimization models in Python using Gurobi 7.5.2 as the solver on an Intel i5-3470T 2.90 GHz machine with 8 GB RAM. 

\subsection{Numerical Results and Analyses}
\label{caseStudySubsection}

We first examine how facility location decisions are affected by the demand variability in  parameter choices, robustness levels, and modeling of decision-dependency in DDDR and other benchmark approaches. In particular, we consider location solutions given by SP, DR, and DDDR models over a set of diverse instances. For obtaining location solutions of a SP model, we generate training samples with 20 or 100 scenarios following a Normal distribution with mean and variance of the demand at each customer site $j \in J$ being $\bar{\mu}_j$ and $\bar{\sigma}_j^2$, respectively. We refer to instances of the two different sizes as SP(20) and SP(100), respectively. For evaluating facility location solutions, we generate 1000 test scenarios for each given solution. In particular, given a solution $\hat{y}$, we first obtain the values of the moment functions $\mu_j(\hat{y})$ and $\sigma^2_j(\hat{y})$ for each customer site $j \in J$ using \eqref{eq:momentFunctions}, and generate test scenarios based on these values following a certain distribution. In the default setting, we test our results by considering Normal distribution as the true representative of the underlying demand distribution but vary the distribution type in one set of tests later. 

Table \ref{objFncUnmetDemandComparisonTableAllReplications} shows the average optimal objective value and unmet demand value over 10 instances evaluated over the test scenarios for five instance sizes under different approaches. Specifically, we consider $|I| \in \{5, \cdots, 10\}$, and $|J| = 2|I|$. Since we minimize the total cost minus revenue, smaller objective values are preferred. The results demonstrate the superior performance of DDDR solutions over the ones of existing methodologies with the decision independent assumption on demand. For instance, the DDDR approach provides, on average, 18\% and 12\% improvement in profit, and 99\% and 96\% reduction in unmet demand, compared to SP and DR approaches over instances with 10 facilities, respectively. Consequently, the decision-independent approaches obtain less profit and provide a lower quality of service by not fully satisfying the demand. 

\begin{table}[h]
\centering
\caption{Average optimal objective and unmet demand values under different methodologies and instances.}
\label{objFncUnmetDemandComparisonTableAllReplications}
\begin{tabular}{crrrrr}
\hline
           &        $|I|$ &    SP(20) &   SP(100) &         DR &       DDDR \\
\hline
\multirow{6}{1.75cm}{\centering average optimal objective} &        5 &   $-12806.7$ &   $-12763.3$ &   $-7618.55$ &   $-22554.2$ \\  
 &          6 &   $-21460.7$ &   $-21483.7$ &   $-16425.1$ &   $-30298.1$ \\
 &          7 &   $-28201.8$ &   $-27810.6$ &   $-24911.6$ &   $-36608.5$ \\
 &          8 &   $-40914.9$ &   $-39206.7$ &   $-35810.4$ &   $-48027.2$ \\
 &          9 &   $-48247.1$ &   $-48790.2$ &   $-51503.5$ &   $-59816.9$ \\
 &        10 &  $-63281.7$ &   $-63337.3$ &   $-67084.4$ &   $-75164.8$ \\ 
\hline
\multirow{6}{1.75cm}{\centering average unmet demand} &  5 &       95.4 &       95.9 &      128.2 &       15.1 \\
&          6 &       75.7 &       74.7 &      110.6 &        3.9 \\
&          7 &       71.7 &       75.1 &       86.2 &        0.0 \\
&          8 &       56.8 &       67.8 &       82.1 &        0.1 \\
&          9 &       58.1 &       53.9 &       38.7 &        0.2 \\
&          10 &       47.0 &       46.9 &       11.1 &        0.4 \\
\hline
\end{tabular}  
\end{table}

We further provide a detailed case study by examining an instance with 10 candidate facility locations, and 20 demand sites. We visualize the customer sites and possible facility locations in Figure \ref{LocationFigureCustFacilities}. The customer sites are shown in circles, and the possible facility locations are denoted in squares. 

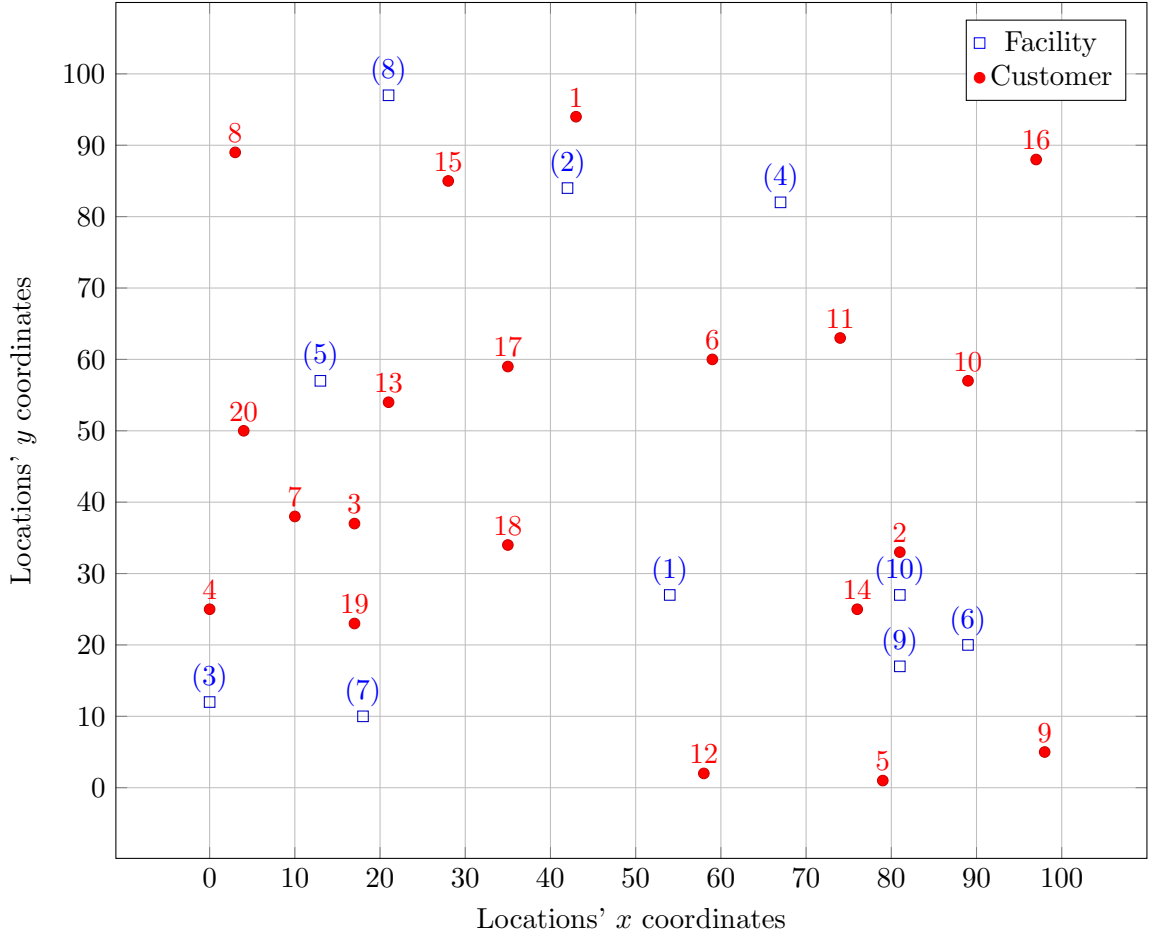
\begin{figure}[h]
\centering
\captionsetup{justification=centering}
\begin{tikzpicture}
	\begin{axis}[
	scale=2.0,
	xlabel={Locations' $x$ coordinates},
		ylabel={Locations' $y$ coordinates},
		ymajorgrids=true,
    xmajorgrids=true,
    xtick = {0,10,20,30,40,50,60,70,80,90,100},
    ytick = {0,10,20,30,40,50,60,70,80,90,100},
    xmax = 110, ymax=110
		]
		\addplot[
		  scatter,mark=square,only marks,
		  point meta=\thisrow{color},
		  color=blue,
		  nodes near coords*={$(\pgfmathprintnumber[frac]\myvalue)$},
		  visualization depends on={\thisrow{myvalue} \as \myvalue},
		] 
		table {
x	y	color 	myvalue
54	27	1	1
42	84	1	2
0	12	1	3
67	82	1	4
13	57	1	5
89	20	1	6
18	10	1	7
21	97	1	8
81	17	1	9
81	27	1	10
		};\addlegendentry{Facility}
		
				\addplot[
				scale=2.0,
		  scatter,mark=*,only marks,
		  color=red,
		  point meta=\thisrow{color},
		  nodes near coords*={$\pgfmathprintnumber[frac]\myvalue$},
		  visualization depends on={\thisrow{myvalue} \as \myvalue},
		] 
		table {
x	y	color 	myvalue
43	94	2	1
81	33	2	2
17	37	2	3
0	25	2	4
79	1	2	5
59	60	2	6
10	38	2	7
3	89	2	8
98	5	2	9
89	57	2	10
74	63	2	11
58	2	2	12
21	54	2	13
76	25	2	14
28	85	2	15
97	88	2	16
35	59	2	17
35	34	2	18
17	23	2	19
4	50	2	20
		};\addlegendentry{Customer};
	\end{axis}
\end{tikzpicture}
\caption{Locations of the customers and potential facilities.}
\label{LocationFigureCustFacilities}
\end{figure}

Then, Table \ref{objFncUnmetDemandComparisonTable} presents the results of DDDR, DR, and SP approaches under the default setting, where we report the average, standard deviation and percentile values of the optimal objective value and unmet demand value over the 1000 test scenarios. Overall, the DDDR approach provides the best results in terms of optimal objective and unmet demand values. DR is better than SP in terms of percentile values of the optimal objective and unmet demand. In addition, SP with different training data sizes present similar results. The table also demonstrates the importance of considering the decision-dependency as the DDDR approach outperforms DR in terms of profit and quality of service. 

\begin{table}[h]
\centering
\caption{Statistics of the optimal objective and unmet demand values under different methodologies  for a specific instance shown in Figure \ref{LocationFigureCustFacilities}.}
\label{objFncUnmetDemandComparisonTable}
\begin{tabular}{crrrr}
\hline
           & SP(20) & SP(100) &         DR &       DDDR \\
           \hline
       average opt.\ objective &   $-53581.0$ &   $-53468.2$ &   $-61443.7$ &   $-64375.0$ \\
  std. dev. &     6457.3 &     6712.0 &     6613.8 &     4917.8 \\
      95\% &   $-42448.0$ &   $-42643.8$ &   $-50474.8$ &   $-55845.2$ \\
      90\% &   $-44968.7$ &   $-44505.1$ &   $-52944.0$ &   $-58082.6$ \\
      75\% &  $-49367.8$ &   $-48939.7$ &   $-57056.9$ &   $-61178.4$ \\
      50\% &   $-53957.6$ &   $-53666.1$ &   $-61504.1$ &   $-64362.4$ \\
\hline
average unmet demand &       59.0 &       61.3 &        3.0 &        0.3 \\
  std. dev. &       35.1 &       35.6 &        6.8 &        2.3 \\
      95\% &      124.4 &      122.8 &       18.3 &        0.0 \\
      90\% &      105.6 &      107.9 &       11.4 &        0.0 \\
      75\% &       80.2 &       81.9 &        2.6 &        0.0 \\
      50\% &       54.4 &       58.0 &        0.0 &        0.0 \\
\hline
\end{tabular}  
\end{table}

\subsubsection{Effect of the variability in demand}
Next, we show how solutions produced by different models are affected by the variability of the underlying demand data. Figure \ref{coeffofVariationFigure} shows average optimal objective and unmet demand values over all the test instances, where the coefficient of variation is used for representing the demand variability. As the coefficient of variation used for estimating empirical mean and variance, namely $\frac{\bar{\sigma}_j}{\bar{\mu}_j}$ for demand at each customer site $j$, increases, the corresponding demand variability increases, assuming that the empirical mean is kept constant. Consequently, the robust approaches (DR and DDDR) become more suitable as compared to SP under higher variability as they obtain location plans that are more reliable to various demand patterns in the test scenarios. Moreover, SP is more sensitive to the underlying variability as the performance of its solutions monotonically worsens as demand variance increases. As the coefficient of variation decreases, the demand variability decreases and the performance of the stochastic and robust approaches become similar to each other. Moreover, the DDDR approach performs significantly better in all settings, highlighting the importance of considering decision dependency in uncertainty quantification. 

\begin{figure}[h]
\centering
\captionsetup{justification=centering}
\begin{subfigure}[b]{.45\textwidth}
  \centering
\begin{tikzpicture}[scale=0.85]
\begin{axis}[
	xlabel={Coefficient of variation (squared)},
	ylabel={Average objective function value},
	ymin = -65500, ymax=-50000,
	xmin = 0.01, xmax = 1,
	xtick = {0.01, 0.01, 0.1, 1},
	ytick = {-65000, -60000, -55000, -50000},
	 xmode = log,
     log basis x=10,
     legend pos=north west,
]
\addplot coordinates {
	(0.01, -58534.3) (0.1, -56365.2) (1, -53581.0)  
};

\addplot coordinates {
	(0.01, -58545.4) (0.1, -56331.3) (1, -53468.2)  
};

\addplot[mark=triangle,color=olive] coordinates {
	(0.01, -61443.7) (0.1, -56305.9) (1, -56824.3)  
};

\addplot coordinates {
	(0.01, -64642.4) (0.1, -64472.9) (1, -64375.0)  
};

\legend{SP(20),SP(100),DR,DDDR}
\end{axis}
\end{tikzpicture}
\end{subfigure}
\begin{subfigure}[b]{.45\textwidth}
  \centering
  \begin{tikzpicture}[scale=0.85]
\begin{axis}[
	xlabel={Coefficient of variation (squared)},
	ylabel={Average unmet demand},
	ymin = -2, ymax=60,
	xmin = 0.01, xmax = 1,
	xtick = {0.01, 0.01, 0.1, 1},
	ytick = {0,10,20,30,40,50,60},
	 xmode = log,
     log basis x=10,
     legend pos=north west,
]
\addplot coordinates {
	(0.01, 5.8) (0.1, 17.9) (1, 59)  
};

\addplot coordinates {
	(0.01, 6.0) (0.1, 18.9) (1, 61.3)  
};

\addplot[mark=triangle,color=olive] coordinates {
	(0.01, 5.8) (0.1, 18.1) (1, 3.0)  
};

\addplot coordinates {
	(0.01, 0.0) (0.1, 0.0) (1, 0.3)  
};

\legend{SP(20),SP(100),DR,DDDR}
\end{axis}
\end{tikzpicture}
\end{subfigure}
\caption{Effect of the variability level of demand on different methodologies.}
\label{coeffofVariationFigure}
\end{figure}
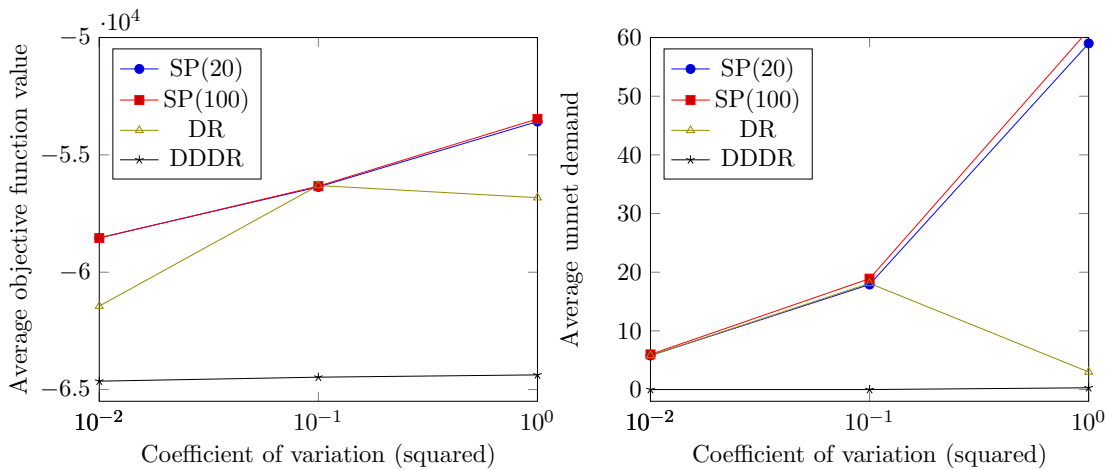

Next, we analyze the effect of misspecifying the true demand distribution by constructing a set of test scenarios, for each solution $\hat{y}$, using a Gamma distribution, where the scale parameter $\hat{\theta}^{\gamma}_j =\sigma^2_j(\hat{y})/\mu_j(\hat{y})$ and the shape parameter $\hat{k}^{\gamma}_j = \mu_j(\hat{y})/\hat{\theta}^{\gamma}_j$ for each customer site $j$. Table \ref{objFncUnmetDemandComparisonTableTestGamma} provides the corresponding results in comparison to Table \ref{objFncUnmetDemandComparisonTable}, where the solutions were tested over test scenarios following a Normal distribution. As Gamma distributions are more skewed, the percentile results worsen for all approaches. Moreover, SP cannot capture the changes in the underlying distribution, whereas DR and DDDR are not much impacted by these changes. The proposed DDDR approach again yields the best results in terms of average, standard deviation and percentile values of the optimal objective and unmet demand. 

\begin{table}[h]
\centering
\caption{Optimal objective and unmet demand values under different methodologies with  Gamma distribution.}
\label{objFncUnmetDemandComparisonTableTestGamma}
\begin{tabular}{crrrr}
\hline
           & SP(20) & SP(100) &         DR &       DDDR \\
\hline
       average opt.\ objective &   $-51962.2$ &   $-51627.2$ &   $-60575.0$ &   $-64268.5$ \\
  std. dev. &     5779.2 &     5951.9 &     6145.7 &     4694.6 \\
      95\% &   $-42495.9$ &   $-41383.7$ &   $-50788.8$ &   $-56590.5$ \\
      90\% &   $-44328.6$ &   $-43867.8$ &   $-52915.7$ &   $-58320.5$ \\
      75\% &   $-48110.2$ &   $-47766.7$ &   $-56590.1$ &   $-61069.6$ \\
      50\% &   $-51858.0$ &   $-51960.7$ &   $-60525.0$ &   $-64306.2$ \\
\hline
average unmet demand &       70.7 &       71.1 &        8.2 &        1.0 \\
  std. dev. &       47.5 &       46.8 &       16.8 &        5.0 \\
      95\% &      160.1 &      157.2 &       38.0 &        5.4 \\
      90\% &      137.4 &      136.3 &       23.9 &        0.0 \\
      75\% &       97.3 &       98.7 &        9.5 &        0.0 \\
      50\% &       61.2 &       62.8 &        0.0 &        0.0 \\
\hline
\end{tabular}
\end{table}

\subsubsection{Effect of unit penalty setting for unmet demand}

We examine the effect of the parameter setting for penalizing each unit of unmet demand. Table \ref{facilityLocationsComparisonTable} shows the facility location plans given by different approaches with unit penalty cost $p_j = 150, \ 225, \ 300$ for all $j \in J$. The case with $p_j = 225$ corresponds to facility location solutions in Table \ref{objFncUnmetDemandComparisonTable}, and $p_j = 150$ represents the case when the penalty parameter is equal to the revenue amount per unit. In the decision-dependent approach, the mean of the underlying demand increases as we open new facilities. Consequently, the DDDR model enforces opening more facility locations yielding higher demand and thus higher revenue. Furthermore, as unit penalty gets higher, it becomes more undesirable to have unmet demand. Thus, all approaches open more facilities when unit penalty cost increases. 

\begin{table}[h]
\centering
\caption{Facility location solutions given by different approaches for different $p_j$-values.}
\label{facilityLocationsComparisonTable}
\begin{tabular}{cccc}
\hline
           & \multicolumn{ 3}{c}{Open facility locations} \\
           \cline{2-4} 
           &  $p_j$ = 150 & $p_j$ = 225 &  $p_j$ = 300 \\
           \cline{2-4}
SP(20) &   1,5,7,10 &   1,5,7,10 & 1,5,6,7,10 \\
SP(100) &   1,5,7,10 &   1,5,7,10 & 1,5,6,7,10 \\
        DR &     1,7,10 & 1,3,5,6,7,10 & 1,3,4,5,6,7,10 \\
      DDDR & 1,4,5,6,7,10 & 1,2,4,5,6,7,9,10 & 1,2,4,5,6,7,9,10 \\
\hline
\end{tabular}  
\end{table}

Figure \ref{penaltyChangeFigure} shows how the average optimal objective and unmet demand values are affected by the changes in the penalty parameter. As the penalty parameter increases, the amount of unmet demand decreases for all approaches, as expected. When penalty parameter takes its smallest value, DR has the worst performance both in the optimal objective and unmet demand values, for which we provide a detailed explanation as follows. The DR approach compares two unfavorable cases: (i) opening many locations but having few customers, and (ii) not opening many locations and missing potential customers. By favoring the latter case, the DR solution loses customers by not having enough facilities open and neglecting the increase in the demand caused by the opening of new facilities. This effect can be also seen in Table \ref{facilityLocationsComparisonTable} as the DR approach opens fewer locations under small penalty values. On the other hand, the DDDR approach outperforms DR and SP in all settings, resulting in better optimal objective value and less unmet demand. 

\begin{figure}[h]
\centering
\captionsetup{justification=centering}
\begin{subfigure}[b]{.45\textwidth}
  \centering
\begin{tikzpicture}[scale=0.85]
\begin{axis}[
	xlabel={Penalty parameter},
	ylabel={Average objective function value},
	ymin = -65500, ymax=-48000,
	xmin = 150, xmax = 300,
	xtick = {150, 225, 300},
	ytick = {-65000, -60000, -55000, -50000},
     legend pos=north west,
]
\addplot coordinates {
	 (150, -58008.3)(225, -53581.0)(300, -57988.5)
};

\addplot coordinates {
	 (150, -58067.8)(225, -53468.2)(300, -57875.2)
};

\addplot[mark=triangle,color=olive] coordinates {
	 (150, -48738.3)(225, -56824.3)(300, -62529.5)
};

\addplot coordinates {
	 (150, -64890.7)(225, -64375.0)(300, -64350.9)
};

\legend{SP(20),SP(100),DR,DDDR}
\end{axis}
\end{tikzpicture}
\end{subfigure}
\begin{subfigure}[b]{.45\textwidth}
  \centering
  \begin{tikzpicture}[scale=0.85]
\begin{axis}[
	xlabel={Penalty parameter},
	ylabel={Average unmet demand},
	xmin = 150, xmax = 300,
	xtick = {150, 225, 300},
	ymin = -2, ymax=115,
	ytick = {0,10,20,30,40,50,60,115},
     legend pos=north west,
]
\addplot coordinates {
	 (150, 59.0) (225, 59.0) (300, 23.6)
};

\addplot coordinates {
	 (150, 61.3) (225, 61.3) (300, 25.1)
};

\addplot[mark=triangle,color=olive] coordinates {
	 (150, 114) (225, 6.9) (300, 0.3)
};

\addplot coordinates {
	(150, 1.6) (225, 0.3) (300, 0.3)
};

\legend{SP(20),SP(100),DR,DDDR}
\end{axis}
\end{tikzpicture}
\end{subfigure}
\caption{Effect of the penalty parameter on different methodologies.}
\label{penaltyChangeFigure}
\end{figure}
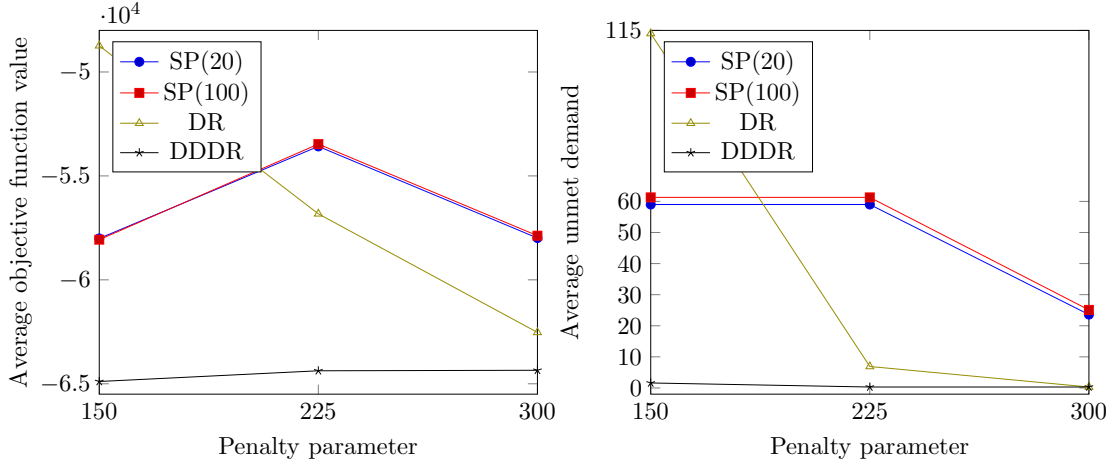

\subsubsection{Effect of the robustness level}

We now examine the effect of the robustness level of the ambiguity set \eqref{eq:ambiguitySet} on facility location solutions. We adjust the parameters $\epsilon^\mu_j$, $\underline{\epsilon}^\sigma_j$, $\overline{\epsilon}^\sigma_j$ for each customer site $j \in J$. Recall that, in the default setting, $\epsilon_j^\mu = 0$, and $\underline{\epsilon}^\sigma_j$ = $\overline{\epsilon}^\sigma_j= 1$ under the assumption of having the perfect knowledge about the underlying mean and variance parameters for each customer site. By adjusting these parameters, we construct models that are robust to different levels of uncertainty in the distribution parameters. 

To evaluate the resulting facility location solutions, we consider a different procedure for generating test scenarios. We first compute $\mu_j(\hat{y})$ and $\sigma^2_j(\hat{y})$ for each customer $j \in J$ given a location solution. Then, we sample the mean and variance parameters from the ranges $[(1-\epsilon^\mu_j)\mu_j(\hat{y}), (1+\epsilon^\mu_j)\mu_j(\hat{y})]$ and $[(1-\underline{\epsilon}^\sigma_j)\sigma_j(\hat{y}), (1+\overline{\epsilon}^\sigma_j)\sigma_j(\hat{y})]$, respectively. After that, we generate 100 Normally distributed scenarios  using the sampled mean and variance parameters. We repeat this procedure ten times to construct the set of test scenarios of size 1000, where each subset of scenarios has its own distribution. 

Table \ref{objFncUnmetDemandEpsilon0.2Table} shows the performance of different solution methodologies under 20\% level of robustness, where the level of robustness $0 \leq \kappa \leq 1$ implies $\epsilon^\mu_j = \kappa \mu_j(y)$, $\underline{\epsilon}^\sigma_j = 1- \kappa$, and $\overline{\epsilon}^\sigma_j = 1 + \kappa$ for every customer site $j$. As the level of robustness parameter $\kappa$ increases, we consider a wider range for the underlying uncertainty. Thus, distributionally robust approaches (i.e., DR and DDDR) become more pre-cautious to the increased ambiguity. On the other hand, SP solutions are not affected by these changes as they are trained with the same data and procedures. Consequently, the distributionally robust approaches perform better than SP under higher $\kappa$-values. Furthermore, the DDDR's results are less affected by the increased robustness, in terms of the optimal objective and unmet demand values, as compared to the default robustness setting, $\kappa = 0$, in Table \ref{objFncUnmetDemandComparisonTable}.

\begin{table}[h]
\centering
\caption{Optimal objective and unmet demand values when the level of robustness $\kappa= 20\%$.}
\label{objFncUnmetDemandEpsilon0.2Table}
\begin{tabular}{crrrr}
\hline
           & SP(20) & SP(100) &         DR &       DDDR \\
\hline
       average opt.\ objective &   $-51974.0$ &   $-52404.0$ &   $-60819.5$ &   $-63572.2$ \\
  std. dev. &     6783.4 &     6819.6 &     7295.4 &     5101.9 \\
      95\% &   $-40260.9$ &   $-41266.0$ &   $-48492.4$ &   $-54978.0$ \\
      90\% &   $-42887.1$ &   $-43644.5$ &   $-51692.9$ &   $-56970.0$ \\
      75\% &   $-47460.6$ &   $-47809.1$ &   $-55968.8$ &   $-60213.8$ \\
      50\% &   $-52170.1$ &   $-52780.4$ &   $-60927.2$ &   $-63639.1$ \\
\hline
average unmet demand &       61.6 &       63.5 &        3.6 &        0.5 \\
  std. dev. &       36.6 &       34.1 &        7.4 &        3.4 \\
      95\% &      129.0 &      127.0 &       20.8 &        1.3 \\
      90\% &      112.7 &      111.9 &       14.0 &        0.0 \\
      75\% &       82.7 &       84.7 &        3.9 &        0.0 \\
      50\% &       56.7 &       59.9 &        0.0 &        0.0 \\
\hline
\end{tabular}  
\end{table}

As the level of robustness increases, the set of test scenarios includes more variability. Due to this increased variability, all approaches have higher standard deviations and worsen percentile values for the optimal objective and unmet demand values over all test scenarios. Despite of this, the distributionally robust approaches (DR and DDDR) under $\kappa = 5\%$ and $10\%$ have the same facility location plans as when $\kappa = 20\%$. 

\subsubsection{Effect of the total number of facilities to open}

We compare solutions of DDDR, DR, and SP given a limit on the total number of facilities to open. We add a constraint to the optimization models, specifically to the polyhedron $\mathcal{Y}$, which restricts the total number of locations that can be selected. Figure \ref{numberofOpenFacilitiesFigure} summarizes the performance of each approach in terms of average optimal objective value and unmet demand value under different budget values on opening facilities. As the DDDR approach considers demand increase given by opening more facilities, adding such a limit hinders its capability of doing so. Consequently, the performance of the solutions of DDDR, DR, SP approaches becomes similar if given smaller facility-opening budget. On the other hand, as we relax this limitation, the DDDR approach outperforms the others, whereas SP is not affected by the relaxation. These results provide us managerial insights for better suitability of the decision-dependent demand distribution modeling in business settings, where there is less restriction on the maximum number of open facilities. 

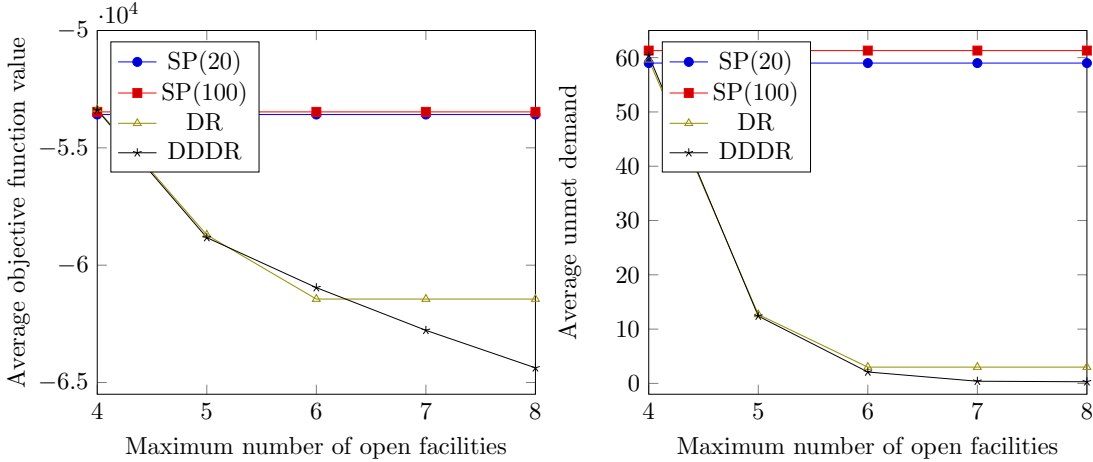
\begin{figure}[h]
\centering
\captionsetup{justification=centering}
\begin{subfigure}[b]{.45\textwidth}
  \centering
\begin{tikzpicture}[scale=0.85]
\begin{axis}[
	xlabel={Maximum number of open facilities},
	ylabel={Average objective function value},
	ymin = -65500, ymax=-50000,
	xmin = 4, xmax = 8,
	xtick = {4,5,6,7,8},
	ytick = {-65000, -60000, -55000, -50000},
     legend pos=north west,
]
\addplot coordinates {
	(4, -53581.0) (5, -53581.0) (6, -53581.0) (7, -53581.0) (8, -53581.0) 
};

\addplot coordinates {
	(4, -53468.2) (5, -53468.2) (6, -53468.2) (7, -53468.2) (8, -53468.2)
};

\addplot[mark=triangle,color=olive] coordinates {
	(4, -53356.9) (5, -58710.0) (6, -61443.7) (7, -61443.7) (8, -61443.7)
};

\addplot coordinates {
	(4, -53391.1) (5, -58825.3) (6, -60957.1) (7, -62776.0) (8, -64375.0)  
};

\legend{SP(20),SP(100),DR,DDDR}
\end{axis}
\end{tikzpicture}
\end{subfigure}
\begin{subfigure}[b]{.45\textwidth}
  \centering
  \begin{tikzpicture}[scale=0.85]
\begin{axis}[
	xlabel={Maximum number of open facilities},
	ylabel={Average unmet demand},
	ymin = -2, ymax=65,
	xmin = 4, xmax = 8,
	xtick = {4,5,6,7,8},
	ytick = {0,10,20,30,40,50,60},
     legend pos=north west,
]
\addplot coordinates {
	(4, 59.0) (5, 59.0) (6, 59.0) (7, 59.0) (8, 59.0) 
};

\addplot coordinates {
	(4, 61.3) (5, 61.3) (6, 61.3) (7, 61.3) (8, 61.3)
};

\addplot[mark=triangle,color=olive] coordinates {
	(4, 59.6) (5, 12.7) (6, 3.0) (7, 3.0) (8, 3.0)
};

\addplot coordinates {
	(4, 60.4) (5, 12.4) (6, 2.1) (7, 0.4) (8, 0.3)
};

\legend{SP(20),SP(100),DR,DDDR}
\end{axis}
\end{tikzpicture}
\end{subfigure}
\caption{Effect of the capacity on number of facilities to open.}
\label{numberofOpenFacilitiesFigure}
\end{figure}

\subsubsection{Effect of the form of decision-dependency}

We examine how DDDR results are affected by the modeling of location dependency on demand distributions. Recall that, in the default setting, we consider $\lambda_{ji}^\mu$, $\lambda_{ji}^\sigma$ for each pair of customer site $j$ and facility location $i$ as decreasing functions of their distance. As a comparison, we propose a clustering-based decision-dependency formulation for modeling the ambiguity set. In particular, in our \textit{$\rho$-means approach}, the demand at the customer site $j$ is equally affected by the opening of the closest $\rho$ facilities in its neighborhood. Let $\Rho^\rho_j$ be the set of $\rho$ facility locations that are closest to customer site $j$; define $\lambda_{ji}^\mu =$ $\lambda_{ji}^\sigma = $ $\frac{1}{\rho}$ for each customer site $j \in J$ and facility location $i \in \Rho^\rho_j$. As the facility locations $i \in I \setminus \Rho^\rho_j$ do not affect the demand at customer site $j$, their corresponding values are set to zero. 

We present the location solutions given by different approaches in Table \ref{FacilityLocationsUnderRhoMeans}. The distance-based approach corresponds to the default setting, and $\rho$-means approach is examined under different $\rho$ values. As all possible facility locations are considered in the distance-based approach with inversely proportional values with respect to their corresponding distances, most facilities are opened in this setting. For $\rho$-means approaches, the set of facilities to be open are affected by the choice of $\rho$. As $\rho$ gets larger, distances between customer and location pairs start to impact the demand less, and other factors such as opening cost of the locations may become more important. We note that $\rho = 10$ corresponds to an extreme case where all facilities equally affect the demand at any customer site. 

\begin{table}[h]
\centering
\caption{Facility location solutions of DDDR with different location-dependency patterns.}
\label{FacilityLocationsUnderRhoMeans}
\begin{tabular}{ccc}
\hline
\multicolumn{2}{c}{Modeling approach} & Open facility locations \\
\hline
distance-based &            & 1,2,4,5,6,7,9,10 \\
\hline
\multirow{5}{*}{$\rho$-means} &          1 & 1,4,5,6,7,8,10 \\
 &          2 & 1,2,3,4,5,7,10 \\
 &          3 & 1,2,3,5,6,7,10 \\
&          5 & 1,2,3,4,5,7,10 \\
& 10 & 1,3,4,5,6,7,10 \\
\hline
\end{tabular}  
\end{table}

\subsection{Results of computational time}
\label{compPerfSection}

Lastly, we compare the solution-time performance of SP, DR and DDDR approaches for different instance sizes. Figure \ref{runtimeComparisonFigure} provides the run time for cases $|I| \in \{5, \cdots, 10\}$, and $|J| = 2|I|$. The run time denotes the average CPU time over 10 different randomly generated instances. In these replications, the default parameter configurations  and moment-based ambiguity sets are used as described in Section \ref{caseStudySubsection}. The distributionally robust approaches are more computationally expensive, whereas SP is the fastest. Furthermore, run time of the DDDR approach is more sensitive to the size of instances, despite of its better performance in terms of cost and demand  satisfaction. Also, the computational  time of DDDR model \eqref{eq:allFormulationv2} depends on the upper bounds of the dual variables, which are set to 100 for all experiments.

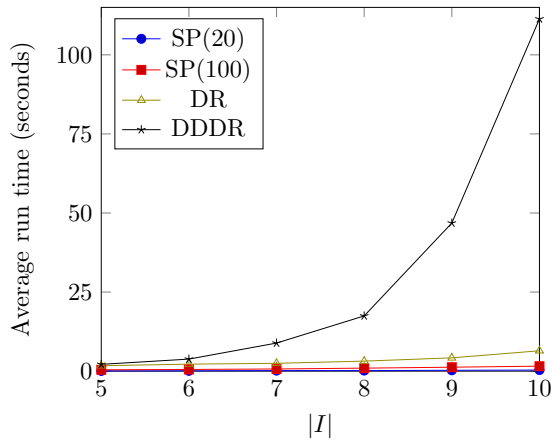
\begin{figure}[h]
  \centering
  \captionsetup{justification=centering}
\begin{tikzpicture}[scale=0.85]
\begin{axis}[
	xlabel={$|I|$},
	ylabel={Average run time (seconds)},
	ymin = 0, ymax=115,
	xmin = 5, xmax = 10,
	xtick = {5,6,7,8,9,10},
	ytick = {0,25,50,75,100},
     legend pos=north west,
]
\addplot coordinates {
	(5, 0.07)(6, 0.10)(7, 0.17)(8,	0.18)(9, 0.27)(10, 0.35)
};

\addplot coordinates {
	(5, 0.43)(6, 0.53)(7,0.69)(8,0.95)(9,1.25)(10,	1.57)
};

\addplot[mark=triangle,color=olive] coordinates {
	(5, 1.71)(6, 	2.21)(7,2.48)(8,	3.17)(9,	4.20)(10,	6.44)
};
\addplot coordinates {
	(5, 2.15)(6,	3.81)(7,	8.88)(8,	17.46)(9,	46.82)(10,	111.36)
};

\legend{SP(20),SP(100),DR,DDDR}
\end{axis}
\end{tikzpicture}
\caption{Run time comparison of different methodologies.}
\label{runtimeComparisonFigure}
\end{figure}


Next we examine the effect of the inclusion of valid inequalities \eqref{eq:validInequalities} to model \eqref{eq:allFormulationv2}. Table \ref{ValidIneqrun timeTable} provides the average run time comparison of two formulations over 10 randomly generated instances of different sizes. We present the speed-ups in comparison to the formulation~\eqref{eq:allFormulationv2} without the valid inequalities \eqref{eq:validConstr1}--\eqref{eq:validConstr3} and the corresponding additional variables and constraints \eqref{eq:validConstr4} and \eqref{eq:validConstr5}. These results illustrate the speed-up due to the proposed inequalities in the order of 3\%--19\% for different instances. 

\begin{table}
\centering
\caption{Effect of the valid inequalities on CPU time results.}
\label{ValidIneqrun timeTable}
\begin{tabular}{crrrrrr}
\hline
           & \multicolumn{6}{c}{$|I| \times |J|$} \\
           \cline{2-7}
           & $5 \times 10$ & $6 \times 12$ & $7 \times 14$ & $8 \times 16$ & $9 \times 18$ & $10 \times 20$ \\
           \hline
    DDDR Average run time (seconds) & 2.15& 3.81	& 8.88 & 	17.46 & 	46.82 & 	111.36 \\
    Speed-up (times) & 1.04 & 1.19 & 1.09 & 1.12 & 1.03 & 1.15 \\
   \hline
\end{tabular}   
\end{table}

\section{Conclusion}
\label{Conclusion}

In this study, we propose a novel framework for modeling the facility location problem under distributionally robust decision-dependent demand distributions. We first provide a moment-based ambiguity set for describing the demand distributions of interest. We define the mean and variance of stochastic demand at each customer location as piecewise linear functions of the facility location decisions. Then, we formulate the  distributionally robust facility location problem under the proposed decision-dependent ambiguity set. We provide a closed-form expression for the inner problem which determines the assignment of customer demand to the open facilities. We further benefit from linear programming duality and convex envelopes to obtain exact representation of the proposed model as a mixed-integer linear program, and derive valid inequalities to further strengthen our formulation. An extensive set of instances are tested to assess the performance of the proposed approach depending on various problem characteristics. Our studies indicate superior performance of the proposed approach, which results in consistently higher profit and less unmet demand, compared to existing stochastic programming and distributionally robust methods. We also present the computational efficiency of the proposed valid inequalities with up to 19\% speed-up for different instances. 
We believe that our study leverages a novel line of research by providing insights for the facility location and optimization under uncertainty literature, and highlighting the need to represent the dependency between customer behavior and planner's decisions within various business settings.

~\\
{\bf Acknowledgement: } The authors are grateful to the support of United States National Science Foundation Grants CMMI-1727618, CCF-1709094, and CMMI-1633196 for this project. 

\bibliographystyle{apa}
\bibliography{references}

\begin{thebibliography}{}

\bibitem[\protect\astroncite{Ahmed}{2000}]{Ahmed2000}
Ahmed, S. (2000).
\newblock {\em Strategic planning under uncertainty: Stochastic integer
  programming approaches}.
\newblock PhD thesis, University of Illinois at Urbana-Champaign.

\bibitem[\protect\astroncite{Basciftci et~al.}{2019a}]{Basciftci2019_Adaptive}
Basciftci, B., Ahmed, S., and Gebraeel, N. (2019a).
\newblock Adaptive two-stage stochastic programming with an application to
  capacity expansion planning.
\newblock \url{https://arxiv.org/abs/1906.03513}.

\bibitem[\protect\astroncite{Basciftci et~al.}{2019b}]{Basciftci2019}
Basciftci, B., Ahmed, S., and Gebraeel, N. (2019b).
\newblock Data-driven maintenance and operations scheduling in power systems
  under decision-dependent uncertainty.
\newblock {\em IISE Transactions}.

\bibitem[\protect\astroncite{Ben-Tal et~al.}{2013}]{Bental2013}
Ben-Tal, A., den Hertog, D., De~Waegenaere, A., Melenberg, B., and Rennen, G.
  (2013).
\newblock Robust solutions of optimization problems affected by uncertain
  probabilities.
\newblock {\em Management Science}, 59(2):341--357.

\bibitem[\protect\astroncite{{Boldrini} et~al.}{2016}]{Boldrini2016}
{Boldrini}, C., {Bruno}, R., and {Conti}, M. (2016).
\newblock Characterising demand and usage patterns in a large station-based car
  sharing system.
\newblock In {\em 2016 IEEE Conference on Computer Communications Workshops
  (INFOCOM WKSHPS)}, pages 572--577.

\bibitem[\protect\astroncite{Chen and Yu}{2016}]{Chen2016}
Chen, A.~Y. and Yu, T. (2016).
\newblock Network based temporary facility location for the emergency medical
  services considering the disaster induced demand and the transportation
  infrastructure in disaster response.
\newblock {\em Transportation Research Part B: Methodological}, 91:408 -- 423.

\bibitem[\protect\astroncite{Ciari et~al.}{2014}]{Ciari2014}
Ciari, F., Bock, B., and Balmer, M. (2014).
\newblock {Modeling station-based and free-floating carsharing demand: Test
  case study for Berlin}.
\newblock {\em Transportation Research Record}, 2416(1):37--47.

\bibitem[\protect\astroncite{Delage and Ye}{2010}]{Delage2010}
Delage, E. and Ye, Y. (2010).
\newblock Distributionally robust optimization under moment uncertainty with
  application to data-driven problems.
\newblock {\em Operations Research}, 58(3):595--612.

\bibitem[\protect\astroncite{Erdo{\u{g}}an and Iyengar}{2006}]{Erdogan2006}
Erdo{\u{g}}an, E. and Iyengar, G. (2006).
\newblock Ambiguous chance constrained problems and robust optimization.
\newblock {\em Mathematical Programming}, 107(1):37--61.

\bibitem[\protect\astroncite{Erlenkotter}{1977}]{Erlenkotter1977}
Erlenkotter, D. (1977).
\newblock Facility location with price-sensitive demands: Private, public, and
  quasi-public.
\newblock {\em Management Science}, 24(4):378--386.

\bibitem[\protect\astroncite{Fischetti et~al.}{2017}]{Fischetti2017}
Fischetti, M., Ljubic, I., and Sinnl, M. (2017).
\newblock Redesigning benders decomposition for large-scale facility location.
\newblock {\em Management Science}, 63(7):2146--2162.

\bibitem[\protect\astroncite{Gao and Kleywegt}{2016}]{Gao2016}
Gao, R. and Kleywegt, A.~J. (2016).
\newblock {Distributionally robust stochastic optimization with Wasserstein
  distance}.
\newblock \url{https://arxiv.org/abs/1604.02199}.

\bibitem[\protect\astroncite{Goel and Grossmann}{2006}]{Goel2006}
Goel, V. and Grossmann, I.~E. (2006).
\newblock A class of stochastic programs with decision dependent uncertainty.
\newblock {\em Mathematical Programming}, 108(2):355--394.

\bibitem[\protect\astroncite{Hellemo et~al.}{2018}]{Hellemo2018}
Hellemo, L., Barton, P.~I., and Tomasgard, A. (2018).
\newblock Decision-dependent probabilities in stochastic programs with
  recourse.
\newblock {\em Computational Management Science}, 15(3):369--395.

\bibitem[\protect\astroncite{Hernández et~al.}{2010}]{Hernandez2010}
Hernández, B., Jiménez, J., and Martín, M.~J. (2010).
\newblock Customer behavior in electronic commerce: The moderating effect of
  e-purchasing experience.
\newblock {\em Journal of Business Research}, 63(9):964 -- 971.

\bibitem[\protect\astroncite{Ho and Perl}{1995}]{PengKuan1995}
Ho, P. and Perl, J. (1995).
\newblock Warehouse location under service-sensitive demand.
\newblock {\em Journal of Business Logistics}, 16(1):133.

\bibitem[\protect\astroncite{Jiang and Guan}{2016}]{Jiang2016}
Jiang, R. and Guan, Y. (2016).
\newblock Data-driven chance constrained stochastic program.
\newblock {\em Mathematical Programming}, 158(1):291--327.

\bibitem[\protect\astroncite{Jorge and Correia}{2013}]{Jorge2013}
Jorge, D. and Correia, G. (2013).
\newblock Carsharing systems demand estimation and defined operations: A
  literature review.
\newblock {\em European Journal of Transport and Infrastructure Research},
  13(3).

\bibitem[\protect\astroncite{Kleywegt et~al.}{2002}]{Kleywegt2002}
Kleywegt, A., Shapiro, A., and Homem-de Mello, T. (2002).
\newblock The sample average approximation method for stochastic discrete
  optimization.
\newblock {\em SIAM Journal on Optimization}, 12(2):479--502.

\bibitem[\protect\astroncite{Lappas and Gounaris}{2017}]{Lappas2017}
Lappas, N.~H. and Gounaris, C.~E. (2017).
\newblock Robust optimization for decision-making under endogenous uncertainty.
\newblock \url{http://www.optimization-online.org/DB_FILE/2017/06/6105.pdf}.

\bibitem[\protect\astroncite{Liu et~al.}{2019}]{Liu2019}
Liu, K., Li, Q., and Zhang, Z. (2019).
\newblock Distributionally robust optimization of an emergency medical service
  station location and sizing problem with joint chance constraints.
\newblock {\em Transportation Research Part B: Methodological}, 119:79 -- 101.

\bibitem[\protect\astroncite{Lu et~al.}{2015}]{Lu2015}
Lu, M., Ran, L., and Shen, Z.~M. (2015).
\newblock Reliable facility location design under uncertain correlated
  disruptions.
\newblock {\em Manufacturing \& Service Operations Management}, 17(4):445--455.

\bibitem[\protect\astroncite{Luo and Mehrotra}{2018}]{Luo2018}
Luo, F. and Mehrotra, S. (2018).
\newblock Distributionally robust optimization with decision dependent
  ambiguity sets.
\newblock \url{https://arxiv.org/abs/1806.09215}.

\bibitem[\protect\astroncite{Mai}{1981}]{Mai1981}
Mai, C. (1981).
\newblock Optimum location and the theory of the firm under demand uncertainty.
\newblock {\em Regional Science and Urban Economics}, 11(4):549 -- 557.

\bibitem[\protect\astroncite{McCormick}{1976}]{McCormick1976}
McCormick, G.~P. (1976).
\newblock Computability of global solutions to factorable nonconvex programs:
  Part i --- convex underestimating problems.
\newblock {\em Mathematical Programming}, 10(1):147--175.

\bibitem[\protect\astroncite{Melo et~al.}{2009}]{Melo2009}
Melo, M., Nickel, S., and da~Gama, F.~S. (2009).
\newblock Facility location and supply chain management – a review.
\newblock {\em European Journal of Operational Research}, 196(2):401 -- 412.

\bibitem[\protect\astroncite{Meyer and Floudas}{2004}]{Meyer2004}
Meyer, C.~A. and Floudas, C.~A. (2004).
\newblock Trilinear monomials with mixed sign domains: Facets of the convex and
  concave envelopes.
\newblock {\em Journal of Global Optimization}, 29(2):125--155.

\bibitem[\protect\astroncite{Mohajerin~Esfahani and Kuhn}{2018}]{Esfahani2018}
Mohajerin~Esfahani, P. and Kuhn, D. (2018).
\newblock Data-driven distributionally robust optimization using the
  wasserstein metric: performance guarantees and tractable reformulations.
\newblock {\em Mathematical Programming}, 171(1):115--166.

\bibitem[\protect\astroncite{Nohadani and Roy}{2017}]{Nohadani2017}
Nohadani, O. and Roy, A. (2017).
\newblock Robust optimization with time-dependent uncertainty in radiation
  therapy.
\newblock {\em IISE Transactions on Healthcare Systems Engineering},
  7(2):81--92.

\bibitem[\protect\astroncite{Nohadani and Sharma}{2018}]{Nohadani2018}
Nohadani, O. and Sharma, K. (2018).
\newblock Optimization under decision-dependent uncertainty.
\newblock {\em SIAM Journal on Optimization}, 28(2):1773--1795.

\bibitem[\protect\astroncite{Noyan et~al.}{2018}]{Noyan2018}
Noyan, N., Rudolf, G., and Lejeune, M. (2018).
\newblock Distributionally robust optimization with decision-dependent
  ambiguity set.
\newblock \url{http://www.optimization-online.org/DB_FILE/2018/09/6821.pdf}.

\bibitem[\protect\astroncite{Owen and Daskin}{1998}]{Owen1998}
Owen, S.~H. and Daskin, M.~S. (1998).
\newblock Strategic facility location: A review.
\newblock {\em European Journal of Operational Research}, 111(3):423 -- 447.

\bibitem[\protect\astroncite{Ozsen et~al.}{2008}]{Ozsen2008}
Ozsen, L., Coullard, C.~R., and Daskin, M.~S. (2008).
\newblock Capacitated warehouse location model with risk pooling.
\newblock {\em Naval Research Logistics (NRL)}, 55(4):295--312.

\bibitem[\protect\astroncite{Popescu}{2007}]{Popescu2007}
Popescu, I. (2007).
\newblock Robust mean-covariance solutions for stochastic optimization.
\newblock {\em Operations Research}, 55(1):98--112.

\bibitem[\protect\astroncite{Royset and Wets}{2017}]{Royset2017}
Royset, J. and Wets, R. (2017).
\newblock Variational theory for optimization under stochastic ambiguity.
\newblock {\em SIAM Journal on Optimization}, 27(2):1118--1149.

\bibitem[\protect\astroncite{S.~Shaheen and Wagner}{1998}]{Shaheen1998}
S.~Shaheen, D.~S. and Wagner, C. (1998).
\newblock Carsharing in europe and north america: Past, present, and future.
\newblock {\em Transportation Quarterly}, 98(3):35 --52.

\bibitem[\protect\astroncite{Santiv{\'a}{\~{n}}ez and
  Carlo}{2018}]{Santivanez2018}
Santiv{\'a}{\~{n}}ez, J.~A. and Carlo, H.~J. (2018).
\newblock Reliable capacitated facility location problem with service levels.
\newblock {\em EURO Journal on Transportation and Logistics}, 7(4):315--341.

\bibitem[\protect\astroncite{Santoso et~al.}{2005}]{Santoso2005}
Santoso, T., Ahmed, S., Goetschalckx, M., and Shapiro, A. (2005).
\newblock A stochastic programming approach for supply chain network design
  under uncertainty.
\newblock {\em European Journal of Operational Research}, 167(1):96 -- 115.

\bibitem[\protect\astroncite{Shaheen et~al.}{2006}]{Shaheen2006}
Shaheen, S., Cohen, A.~P., and Roberts, J.~D. (2006).
\newblock Carsharing in north america: Market growth, current developments, and
  future potential.
\newblock {\em Transportation Research Record}, 1986(1):116--124.

\bibitem[\protect\astroncite{Snyder}{2006}]{Snyder2006}
Snyder, L.~V. (2006).
\newblock {Facility location under uncertainty: A review}.
\newblock {\em IIE Transactions}, 38(7):547--564.

\bibitem[\protect\astroncite{Spacey et~al.}{2012}]{Spacey2012}
Spacey, S.~A., Wiesemann, W., Kuhn, D., and Luk, W. (2012).
\newblock Robust software partitioning with multiple instantiation.
\newblock {\em INFORMS Journal on Computing}, 24(3):500--515.

\bibitem[\protect\astroncite{{Vayanos} et~al.}{2011}]{Vayanos2011}
{Vayanos}, P., {Kuhn}, D., and {Rustem}, B. (2011).
\newblock Decision rules for information discovery in multi-stage stochastic
  programming.
\newblock In {\em 2011 50th IEEE Conference on Decision and Control and
  European Control Conference}, pages 7368--7373.

\bibitem[\protect\astroncite{Vine et~al.}{2011}]{LeVine2011}
Vine, S.~L., Lee-Gosselin, M., Sivakumar, A., and Polak, J. (2011).
\newblock Design of a strategic-tactical stated-choice survey methodology using
  a constructed avatar.
\newblock {\em Transportation Research Record}, 2246(1):55--63.

\bibitem[\protect\astroncite{Zhang et~al.}{2016}]{Zhang2016}
Zhang, J., Xu, H., and Zhang, L. (2016).
\newblock Quantitative stability analysis for distributionally robust
  optimization with moment constraints.
\newblock {\em SIAM Journal on Optimization}, 26(3):1855--1882.

\bibitem[\protect\astroncite{Zhang et~al.}{2015}]{Zhang2015}
Zhang, Z., Berenguer, G., and Shen, Z.~M. (2015).
\newblock A capacitated facility location model with bidirectional flows.
\newblock {\em Transportation Science}, 49(1):114--129.

\end{thebibliography}

\end{document}